\documentclass[12pt]{amsart}
\usepackage{amssymb,upref,enumerate}
\usepackage{tikz}
\usepackage{amsmath,amssymb,amscd,amsthm,amsxtra}
\usepackage{latexsym}

\usepackage[colorlinks=true, pdfstartview=FitV, linkcolor=blue,
citecolor=blue, urlcolor=blue]{hyperref}

\usepackage{pdfsync}

\def\Xint#1{\mathchoice 
{\XXint\displaystyle\textstyle{#1}}%
{\XXint\textstyle\scriptstyle{#1}}%
{\XXint\scriptstyle\scriptscriptstyle{#1}}%
{\XXint\scriptscriptstyle\scriptscriptstyle{#1}}%
\!\int} 
\def\XXint#1#2#3{{\setbox0=\hbox{$#1{#2#3}{\int}$} 
\vcenter{\hbox{$#2#3$}}\kern-.5\wd0}} 
 
\def\dashint{\Xint-}

\newcommand{\ra}{\rightarrow}

\newcommand{\bey}{\begin{eqnarray*}}
\newcommand{\eey}{\end{eqnarray*}}
\newcommand{\ba}{\begin{align}}
\newcommand{\ea}{\end{align}}
\newcommand{\bea}{\begin{align*}}
\newcommand{\ena}{\end{align*}}
\newcommand{\be}{\begin{equation}}
\newcommand{\ee}{\end{equation}}
\newcommand{\R}{\mathbb R}
\newcommand{\Z}{\mathbb Z}

\newcommand{\D}{\mathcal D}

\newcommand{\M}{\mathcal M }

\newcommand{\ep}{\epsilon}
\newcommand{\dss}{\displaystyle}
\newcommand{\bc}{\begin{center}}
\newcommand{\ec}{\end{center}}

\newcommand{\vp}{\varphi}

\newcommand{\al}{\alpha}
\newcommand{\Tin}{T^{d}_{\rm in}}
\newcommand{\Tout}{T^{d}_{\rm out}}
\newcommand{\Msharp}{M^{\sharp, d}_{\lambda, Q}}
\newcommand{\Msharpquar}{M^{\sharp, d}_{\frac{1}{4}, Q}}

\DeclareMathOperator{\supp}{supp}

\newtheorem{theorem}{Theorem}[section]
\newtheorem{lemma}[theorem]{Lemma}

\newtheorem{conjecture}[theorem]{Conjecture}

\allowdisplaybreaks

\theoremstyle{definition}

\newtheorem{definition}[theorem]{Definition}
\theoremstyle{remark}
\newtheorem{remark}[theorem]{Remark}

\numberwithin{equation}{section}

\begin{document}


\subjclass[2000]{Primary 42A50}

\keywords{Commutators, two-weight inequalities, sharp weighted bounds}

\thanks{The first author is supported by the Stewart-Dorwart Faculty Development Fund
  at Trinity College and 
  grant MTM2009-08934 from the Spanish Ministry of Science and
  Innovation}

\address{David Cruz-Uribe, SFO\\
Department of Mathematics\\
Trinity College, Hartford, CT 06106, USA.}\email{david.cruzuribe@trincoll.edu}

\address{Kabe Moen\\
Department of Mathematics\\
University of Alabama, Tuscaloosa, AL, 35487, USA.}\email{kmoen@as.ua.edu}

\title[Inequalities for commutators]{Sharp norm inequalities for commutators of classical operators}
\author{David Cruz-Uribe, SFO and Kabe Moen}
\date{September 12, 2011}

\begin{abstract} 
  We prove several sharp weighted norm inequalities for commutators of
  classical operators in harmonic analysis.  We find sufficient
  $A_p$-bump conditions on pairs of weights $(u,v)$ such that $[b,T]$,
  $b\in BMO$ and $T$ a singular integral operator (such as the Hilbert
  or Riesz transforms), maps $L^p(v)$ into
  $L^p(u)$.  Because of the added degree of singularity, the
  commutators require a ``double log bump" as opposed to that
  of singular integrals, which only require single log bumps.
  For the fractional integral operator $I_\al$ we find the sharp
  one-weight bound on $[b,I_\al]$, $b\in BMO$, in terms of the
  $A_{p,q}$ constant of the weight.  We also prove sharp two-weight
  bounds for $[b,I_\al]$ analogous to those of singular integrals.  We
  prove two-weight weak type inequalities for $[b,T]$ and $[b,I_\al]$
  for pairs of factored weights.  Finally we construct several
  examples showing our bounds are sharp.
\end{abstract}

\maketitle

\bigskip

\section{Introduction}
Given a linear operator $T$ defined on the set of measurable
functions and a function $b$, we define the commutator $[b,T]$ to be
the operator
$$[b,T]f(x)=b(x)Tf(x)-T(bf)(x).$$
Commutators of singular integral operators were introduced by Coifman,
Rochberg, and Weiss \cite{CRW}, who used them to extend the classical
factorization theory of $H^p$ spaces.  They proved that if $b\in
BMO$, then $[b,T]$ is bounded on $L^p(\R^n)$,  $1<p<\infty$. 
Janson \cite{J} later showed the converse:  if $[b,T]$ is bounded, then $b\in BMO$.

Given $0<\al<n$, define the fractional integral operator $I_\al$ by
$$I_\al f(x)=\int_{\R^n} \frac{f(y)}{|x-y|^{n-\al}}\, dy.$$
The commutator $[b,I_\alpha]$ was first
considered by Chanillo~\cite{SC}, who showed that if $b\in BMO$,
$[b,I_\al]$ maps $L^p(\R^n)$ into $L^q(\R^n)$, where
$1/p-1/q={\al}/{n}$;  a dyadic version of this result and
further applications were given by Lacey \cite{Lac}.

While commutators share the same $L^p$ bounds as the
underlying operators (e.g., singular integrals are bounded on $L^p$
and fractional integrals map $L^p$ into $L^q$), they are,
nevertheless, more singular.  This fact was first observed by
considering their behavior at the endpoint.  For instance, a singular
integral operator $T$ is bounded from $L^1(\R^n)$ to
$L^{1,\infty}(\R^n)$, but  $[b,T]$, $b\in BMO$, is not.  Instead, it
satisfies a weaker modular
inequality,
$$|\{x\in \R^n: |[b,T]f(x)|>\lambda\}|\leq 
C\|b\|_{BMO} \int_{\R^n}\Phi\Big( \frac{|f(y)|}{\lambda}\Big)\, dy,$$
where $\Phi(t)=t\log(e+t)$.   See P\'erez~\cite{P1}.   A similar
result holds for fractional integrals;  see~\cite{CF}.  

The greater degree of singularity of commutators is also reflected in
the differences between the sharp weighted norm inequalities for a
commutator and the underlying operator.  This was first shown in a
recent paper by Chung, P\'erez, and Pereyra \cite{CPP}.   (See also
Chung~\cite{C-phd,Ch}.)  To state
their result, recall that for $1<p<\infty$ we say that $w$ is an $A_p$
weight (or, more simply, $w\in A_p$) if 
$$[w]_{A_p}=\sup_Q\left(\dashint_Q w(x)\, dx\right)\left(\dashint_Q
  {w(x)^{1-p'}}\, dx\right)^{p-1}<\infty,$$ 
where the supremum is taken over all cubes $Q$ with sides parallel to
the coordinate axes. 
If $T$ is a singular integral operator, then 
\[\|T\|_{L^p(w)\ra L^p(w)}\leq c[w]_{A_p}^{\max\left(1,\frac{p'}{p}\right)}, \]
and this estimate is sharp in that 
$\max(1,{p'}/{p})$ cannot be replaced by any smaller
power. 
(This result has a long history and has only recently been proved in
full generality.   See~\cite{cruz-uribe-martell-perez2010,CMP3,H,HPTV} for
details and further references.)
However, Chung, P\'erez, and Pereyra  showed that if $b\in BMO$, then 
$$\|[b,T]\|_{L^p(w)\ra L^p(w)} \leq c [w]_{A_p}^{2\max\left(1,\frac{p'}{p}\right)},$$
and this exponent is again sharp.  

In this paper we continue the study of weighted norm inequalities for
commutators.  We prove two-weight, strong type norm inequalities for commutators of
singular integrals and one and two-weight strong type norm inequalities for
commutators of fractional integrals.  In both cases the results we get
are sharp, and (like the result of Chung, Pereyra and P\'erez) they
demonstrate that commutators are more singular than the underlying
operators.  We also consider two-weight,
weak type inequalities for both operators and prove results for a
special class of weights, the so-called factored weights (which we
will define below).   These results are of interest because they
strongly suggest what the sharp results should be, and we make two
conjectures.  

\subsection*{Singular integrals}
We first consider singular integral operators.  Because of our
approach, our proofs are restricted to singular integral operators
that can be approximated by ``dyadic'' singular integral operators
that are generalizations of the Haar shifts.  (Precise definitions
will be given in Section~\ref{section:prelimaries} below.)  Such
operators include the classical singular integrals: the Hilbert
transform, Riesz transforms, and the Ahlfors-Beurling operator.  In
one dimension it also includes any convolution type singular integral
whose kernel is $C^2$: see Vagharshakyan~\cite{V}.  However, in light
of recent results~\cite{H,HPTV} we
conjecture that Theorem~\ref{main} below is true for any
Calder\'on-Zygmund singular integral.

Before stating our result for commutators, we provide some context.
It has long been known that the two-weight $A_p$ condition is not
sufficient for two-weight norm inequalities for singular integrals:
see Muckenhoupt and Wheeden~\cite{muckenhoupt-wheeden76}.   An
important substitute is the so-called $A_p$-bump condition,
\[  \sup_Q \|u^{1/p}\|_{A,Q} \|v^{-1/p}\|_{B,Q} < \infty, \]
where $A,\,B$ are Young functions and
$\|\cdot\|_{A,Q},\,\|\cdot\|_{B,Q}$ are normalized Luxemburg norms on
the cubes $Q$.  (Precise definitions are given below.)  These
conditions have been extensively studied:  see~\cite{CMP2, CMP3, CMP4,
  CP1,CP2,CP3}.    Like the Muckenhoupt $A_p$ weights, these weight classes
have two advantageous features.  First, the $A_p$-bump condition
is ``universal'': it applies simultaneously to large families of
operators.  Second, it is straightforward to check that a given pair
satisfies the condition or to construct a pair of weights that does or
does not satisfy it.

For the class of singular integrals we are
concerned with, the best result is the following.

\begin{theorem}[\cite{CMP2,cruz-uribe-martell-perez2010,CMP3}] \label{thm:sharp-sio}
Given $p$, $1<p<\infty$, suppose $(u,v)$ is a pair of weights such that 
\begin{equation} \label{eqn:sharp-sio1}
\sup_Q \|u^{1/p}\|_{A,Q} \|v^{-1/p}\|_{B,Q} < \infty, 
\end{equation}
where $A(t)=t^p\log(e+t)^{p-1+\delta}$,
$B(t)=t^{p'}\log(e+t)^{p'-1+\delta}$ for some $\delta>0$. If $T$ is any singular
integral that can be approximated by dyadic singular integrals (in
particular, if $T$ is the Hilbert transform, a Riesz transform, or the
Ahlfors-Beurling operator), then 
\[ \| Tf\|_{L^p(u)} \leq c\|f\|_{L^p(v)}. \]
Further, this result is sharp in the sense that if $\delta=0$, then it
does not hold in general.
\end{theorem}

Theorem~\ref{thm:sharp-sio} was proved in~\cite{CMP2} for the Hilbert
transform, and was proved in general
in~\cite{cruz-uribe-martell-perez2010,CMP3}. 

\begin{remark}
  Here and in subsequent theorems, our hypotheses can be stated in
  greater generality, replacing the ``log-bumps'' (as Young functions
  like $A$ and $B$ are generally called) by more general Young
  functions determined by the so-called $B_p$ condition; see
  Definition~\ref{defn:Bpcond}.  However, for commutators it is most
  natural to state our results in this form.  For a brief description
  of a more general formulation, see Remark~\ref{rem:general-bump} below.
\end{remark}

We can now state our main result for commutators of singular
integrals.

\begin{theorem} \label{main} 

 Given $p$, $1<p<\infty$, suppose $(u,v)$ is a pair of weights that satisfies
\begin{equation}\label{twowcond} 
\sup_Q\|u^{1/p}\|_{A,Q} \|v^{-1/p}\|_{B,Q}<\infty,
\end{equation}
where $A(t)=t^p\log(e+t)^{2p-1+\delta}$ and
$B(t)=t^{p'}\log(e+t)^{2p'-1+\delta}$, $\delta>0$.  
If $T$ is any singular
integral that can be approximated by dyadic singular integrals (in
particular, if $T$ is the Hilbert transform, a Riesz transform, or the
Ahlfors-Beurling operator) and $b\in BMO$, then
\begin{equation}\label{2weightcom}
\|[b,T]f\|_{L^p(u)}\leq c\|b\|_{BMO} \|f\|_{L^p(v)}.
\end{equation}
Further, this result is sharp in the sense that if $\delta=0$, then
inequality~\eqref{2weightcom} does not hold in general.
\end{theorem}

\begin{remark} The constant in inequality~\eqref{2weightcom} depends on
  the constant from the condition on the weights in~\eqref{twowcond},
  and in fact this dependence is linear.  This follows from a general scaling principle
  for two-weight inequalities first observed by Sawyer~\cite{S}.   Let
$$[u,v]_{p,A,B}=\sup_Q\|u^{1/p}\|_{A,Q} \|v^{-1/p}\|_{B,Q};$$
then by Theorem 1.3 we have
\begin{equation} \label{depend} \|[b,T]f\|_{L^p(u)}\leq \vp([u,v]_{p,A,B}) \|b\|_{BMO} \|f\|_{L^p(v)}\end{equation}
for some positive function $\vp$.  We now exploit the fact that a two
weight norm inequality has two degrees of freedom:  for any $s,t>0$, 
$$[su,tv]_{p,A,B}=s^{1/p}t^{-1/p}[u,v]_{p,A,B}.$$
Hence, if we substitute $(u,v)\mapsto (su,tv)$ in inequality
\eqref{depend},  we get
$$s^{1/p}\|[b,T]f\|_{L^p(u)}\leq \vp(s^{1/p}t^{-1/p}[u,v]_{p,A,B}) t^{1/p}\|b\|_{BMO} \|f\|_{L^p(v)}.$$
Let  $t=[u,v]_{p,A,B}^p$ and $s=1$; this gives us
$$\|[b,T]f\|_{L^p(u)}\leq \vp(1) [u,v]_{p,A,B}\|b\|_{BMO}
\|f\|_{L^p(v)}$$
which is the desired linear bound.  
\end{remark}

\bigskip

The higher degree of singularity of the commutators is reflected in
the power on the logarithms in the definition of $A$ and $B$: roughly
twice as large as for a singular integral.  (For this reason, we say
that the commutator requires ``double log bumps.'')  The phenomenon of
having the degree of singularity reflected in the power of the
logarithm was first conjectured in~\cite{CMP4} for the dyadic square
function and the vector-valued maximal operator, and confirmed
in~\cite{CMP3}.

Theorem~\ref{main} generalizes a number of known results for
commutators of singular integrals.  \'Alvarez {\em et al.}~\cite{ABKP}
showed that if $W$ is any class of weights that is stable---i.e., if
$(u,v)\in W$, there exists $r>1$ such that $(u^r,v^r)\in
W$---then given any pair $(u,v)\in W$, $[b,T]:L^p(v)\ra L^p(u)$.  The main example
of a class of stable weights consists of pairs $(u,v)$ that
satisfy~\eqref{twowcond} when $A(t)=t^{rp}$ and $B(t)=t^{rp'}$, $r>1$.
This class has the remarkable property that given any such pair
$(u,v)$, there exists $w\in A_p$ such that $c_1u\leq w \leq c_2v$.
See Neugebauer~\cite{neugebauer83} (also see~\cite{CMP4}).
In \cite{CP3}, Theorem~\ref{main} was proved with $A(t)=t^{rp}$,
$r>1$, $B(t)=t^{p'}\log(e+t)^{2p'-1+\delta}$; this was improved
in~\cite{cruz-uribe-fiorenza02} where it was proved with
$A(t)\approx t^p \exp([\log(t^p)]^r)$, $0<r<1$.  Finally,
in~\cite{CMP2} Theorem~\ref{main} was proved with
$A(t)=t^p\log(e+t)^{3p-1+\delta}$.

In~\cite{CMP2}, the condition \eqref{twowcond} was conjectured as
being sufficient for the commutator of any
Calder\'on-Zygmund singular integral operator, and Theorem~\ref{main}
is substantial evidence for this conjecture.   There were two
motivations for this conjecture.  First, it is a natural
generalization of an old (and still outstanding) conjecture of
Muckenhoupt and Wheeden.  They conjectured that given a pair of
weights $(u,v)$, a sufficient condition
for a singular integral to map $L^p(v)$ into $L^p(u)$ is that the
Hardy-Littlewood maximal operator satisfy
\begin{equation} \label{eqn:MW-conj}
M : L^p(v) \rightarrow L^p(u), \qquad M : L^{p'}(u^{1-p'})
\rightarrow L^{p'}(v^{1-p'}).
\end{equation}
The maximal
operator naturally associated with commutators is not $M$, but the
Orlicz maximal operator $M_{L\log L}$ (defined below); therefore, it seems natural to
conjecture that if we replace $M$ by $M_{L\log L}$ in
\eqref{eqn:MW-conj} then we get a sufficient condition for $[b,T] :
L^p(v)\rightarrow L^p(u)$.  The bump condition \eqref{twowcond} is
sufficient for $M_{L\log L}$ to satisfy these two estimates (this
follows from Theorem~\ref{Bpmax} below).  

A second motivation for this conjecture is that for the special class
of factored weights we could readily prove a result that was nearly
optimal.  We will consider this approach more carefully below.  

\subsection*{Fractional integrals}
We can prove both one and two-weight results for commutators of
fractional integrals.    In the one weight case the appropriate class
of weights is $A_{p,q}$, a generalization of the $A_p$ weights
introduced by Muckenhoupt and Wheeden~\cite{MW}.  More precisely,
given $\alpha$, $0<\alpha<n$, and $p$, $1<p<n/\alpha$, fix $q$ so that
$1/p-1/q=\alpha/n$.  Then $w\in A_{p,q}$ if 
\[ [w]_{A_{p,q}}=\sup_Q \left(\dashint_Q w(x)^q \, dx\right)
\left(\dashint_Q w(x)^{-p'} \, dx\right)^{q/p'}<\infty. \]
There is a close connection between $A_{p,q}$ weights and $A_p$
weights: it is immediate from the definition that $[w]_{A_{p,q}}=[w^q]_{A_{1+q/p'}}$.  

If $w\in A_{p,q}$, then $I_\alpha : L^p(w^p) \ra L^q(w^q)$, and 
in \cite{LMPT} the sharp constant in this inequality was given:
\begin{equation}\label{sharpfrac} 
\|I_\al\|_{L^p(w^p)\ra L^q(w^q)}\leq
c[w]_{A_{p,q}}^{(1-\frac{\al}{n})\max\left(1,\frac{p'}{q}\right)}.
\end{equation}
(A local version of this result was proved in \cite{ACS}.)
Our next theorem is the corresponding result for commutators.

\begin{theorem}\label{sharpcom} 
Given $\alpha$,  $0<\al<n$, and $p$, $1<p<n/\al$, fix $q$ such that
$1/p-1/q=\alpha/n$.  Then for any $b\in BMO$ and any $w\in A_{p,q}$,
$[b,I_\alpha] : L^p(w^p) \rightarrow L^q(w^q)$, and 
\begin{equation}\label{sharpbound}
\|[b,I_\al]\|_{L^p(w^p)\ra L^q(w^q)}\leq
c\|b\|_{BMO}[w]_{A_{p,q}}^{(2-\frac{\al}{n})\max\left(1,\frac{p'}{q}\right)}.
\end{equation}
Further, this result is sharp since
$(2-{\al}/{n})\max(1,{p'}/{q})$ cannot be replaced by a smaller power.
\end{theorem}

The restriction $1/p-1/q=\al/n$ in the one-weight case follows from
homogeneity: see \cite[Section 5.6]{CMP4}.  However, in the two-weight
case, since the weights $u$ and $v$ may have different homogeneity,
there is no corresponding restriction.  P\'erez \cite{perez94} proved
that if $1<p\leq q<\infty$, and if the pair $(u,v)$ satisfies
\[ \sup_Q|Q|^{\frac{\al}{n}+\frac1q-\frac1p}\|u^{1/q}\|_{A,Q}
\|v^{-1/p}\|_{B,Q}<\infty, \]
where $A(t)=t^q\log(e+t)^{q-1+\delta}$ and
$B(t)=t^{p'}\log(e+t)^{p'-1+\delta}$, then $I_\alpha :
L^p(v)\rightarrow L^q(u)$.    Given this estimate, our next result is
the natural analog of Theorem~\ref{main} for commutators of fractional integrals.

\begin{theorem} \label{fraccomm} Given $\alpha$, $0<\al<n$, and
  $p,\,q$, $1<p\leq q<\infty$, suppose the pair of weights $(u,v)$ 
satisfies
\begin{equation}\label{twowcondfrac} 
\sup_Q|Q|^{\frac{\al}{n}+\frac1q-\frac1p}\|u^{1/q}\|_{A_q,Q}
\|v^{-1/p}\|_{B,Q}<\infty,
\end{equation}
where $A_q(t)=t^q\log(e+t)^{2q-1+\delta}$ and
$B(t)=t^{p'}\log(e+t)^{2p'-1+\delta}$.  Then for all  $b\in BMO$,
\[ \|[b,I_\al]f\|_{L^q(u)}\leq 
c\|b\|_{BMO} \|f\|_{L^p(v)}. \]
Further, this inequality is sharp since it does not hold in general if we take $\delta =0$ in the definition of $A_q$.
\end{theorem}

As this paper was being completed, we discovered that the sufficiency
of \eqref{twowcondfrac} in Theorem~\ref{fraccomm} was proved earlier by
Li~\cite{li2006}, who adapted the proof of the two-weight norm
inequalities for $I_\alpha$.  Here we give a somewhat more elementary
proof along with an example to show that this condition is sharp.

\medskip

Though not directly connected with our results on commutators, we
digress to give a sharp constant result for the weighted Sobolev
inequality.   In \cite{LMPT} the authors used their results for fractional
integrals to show that for $p,\,q$ such that  $1\leq p<n$ and
$1/p-1/q=1/n$, 
\begin{equation}\label{sobolev}
\|f\|_{L^q(w^q)}\leq c [w]_{A_{p,q}}^{1/n'} \|\nabla f \|_{L^p(w^p)}.
\end{equation} 
Here we show that this inequality is the best possible.

\begin{theorem} \label{sharpsobolev} 
Suppose $n>1$, $1\leq p<n$ and $1/p-1/q=1/n$, then inequality
\eqref{sobolev} is sharp since the exponent $1/n'$ cannot be replaced
by any smaller power.
\end{theorem}

To show that \eqref{sobolev} is sharp we cannot use the standard
examples of the form  $f(x)=|x|^a\chi_B(x)$ where $B$ is a unit ball or
unit cube, since \eqref{sobolev} requires $f$ to be smooth.  We instead
introduce a new family which is smooth and decays exponentially at
infinity.  

\subsection*{Weak type inequalities}
We begin with our two conjectures for weak type inequalities for
commutators.

\begin{conjecture} \label{conj:weak-sio}
Given a Calder\'on-Zygmund singular integral  operator $T$, if for
some $p$, $1<p<\infty$, the pair of weights $(u,v)$ satisfies
\begin{equation} \label{eqn:weak-sio1}
\sup_Q \|u^{1/p}\|_{A,Q}\|v^{-1/p}\|_{B,Q} < \infty, 
\end{equation}
where $A(t)=t^p\log(e+t)^{2p-1+\delta}$, $\delta>0$,
$B(t)=t^{p'}\log(e+t)^{p'}$, then for any $b\in BMO$,
\begin{equation} \label{eqn:weak-sio2}
 [b,T] : L^p(v) \rightarrow L^{p,\infty}(u).
\end{equation}
\end{conjecture}

\begin{conjecture} \label{conj:weak-frac}
Given $\alpha$, $0<\alpha<n$,  if for
some $p$, $1<p<\infty$, the pair of weights $(u,v)$ satisfies
\begin{equation} \label{eqn:weak-frac1}
\sup_Q |Q|^{\alpha/n}\|u^{1/p}\|_{A,Q}\|v^{-1/p}\|_{B,Q} < \infty, 
\end{equation}
where $A(t)=t^p\log(e+t)^{2p-1+\delta}$, $\delta>0$,
$B(t)=t^{p'}\log(e+t)^{p'}$, then for any $b\in BMO$,
\begin{equation} \label{eqn:weak-frac2}
  [b,I_\alpha] : L^p(v) \rightarrow L^{p,\infty}(u).
\end{equation}
\end{conjecture}

Conjecture~\ref{conj:weak-sio} was proved in~\cite{CP2} when
$A(t)=t^{rp}$, $r>1$; in \cite{cruz-uribe-fiorenza02} this was
improved to $A(t) \approx t^p\exp([\log(t^p)]^r)$, $0<r<1$.
Conjecture~\ref{conj:weak-frac} was proved by Liu and
Lu~\cite{liu-lu2004}, again when $A(t)=t^{rp}$, $r>1$; they did so by
adapting the argument in~\cite{CP2} to the case of fractional
integrals.  By combining their proof with the ideas in
\cite{cruz-uribe-fiorenza02}, we get  that this
conjecture is also true with $A(t)\approx t^p\exp([\log(t^p)]^r)$, $0<r<1$.

By comparison, a singular integral $T$ satisfies $T :
L^p(v)\rightarrow L^{p,\infty}(u)$ if the pair $(u,v)$ satisfies
\eqref{eqn:weak-sio1} with $A(t)=t^p\log(e+t)^{p-1+\delta}$ and
$B(t)=t^p$ (see~\cite{CP1}), and it is conjectured that $I_\alpha$
satisfies a weak $(p,p)$ inequality if the pair $(u,v)$ satisfies
\eqref{eqn:weak-frac1} with this same pair of Young functions.  (See
\cite{CMP4}.)

We cannot prove either conjecture; however, we can prove two results
for a special class of weights that strongly suggests that these
conjectures are true.  We consider the so-called factored weights:
pairs of the form 
\[ (w_1(M_\Psi w_2)^{1-p},
(M_\Phi w_1)w_2^{1-p} ), \]
where $M_\Phi$ and $M_\Psi$ are Orlicz maximal operators (which are
defined in Section~\ref{section:prelimaries} below).  Such pairs
are a generalization of the pairs $(u,Mu)$ that have appeared in many
contexts.  Their explicit structure can be combined with
Calder\'on-Zygmund decomposition arguments to prove a variety of
weighted norm inequalities.  In addition, their factored form (which
is in some sense a two-weight version of the Jones' factorization
theorem) makes it straightforward to construct examples of pairs of
weights that satisfy $A_p$ bump conditions.  
Factored weights were introduced and studied systematically
in \cite{CMP4}.

\begin{theorem} \label{thm:factored-sio}
Given a Calder\'on-Zygmund singular integral operator $T$ and $p$,
$1<p<\infty$, then for any pair of non-negative, locally integrable
functions $w_1,\,w_2$, the pair of weights
\[ (\tilde{u},\tilde{v}) = (w_1(M_\Psi w_2)^{1-p},
(M_\Phi w_1)w_2^{1-p} ) \]
where $\Phi(t)=t\log(e+t)^{2p+\delta}$, $\delta>0$, $\Psi(t)=t\log(e+t)^{p'+1}$,
satisfies~\eqref{eqn:weak-sio1} with $A(t) =t^p\log(e+t)^{2p+\delta}$,
and $B(t)=t^{p'}\log(e+t)^{p'+1}$, and for any $b\in BMO$
the commutator $[b,T]$ satisfies \eqref{eqn:weak-sio2}.
\end{theorem}

In the next result, $M_{\Phi,\alpha}$ and $M_{\Psi,\alpha}$ are
fractional Orlicz maximal operators; these will be defined in
Section~\ref{section:prelimaries} below.

\begin{theorem} \label{thm:factored-frac}
Given $\alpha$, $0<\alpha<n$, and $p$,
$1<p<\infty$, then for any pair of non-negative, locally integrable
functions $w_1,\,w_2$, the pair of weights
\[ (\tilde{u},\tilde{v}) = (w_1(M_{\Psi,\alpha} w_2)^{1-p},
(M_{\Phi,\alpha} w_1)w_2^{1-p} ) \]
where $\Phi(t)=t\log(e+t)^{2p+\delta}$, $\delta>0$, $\Psi(t)=t\log(e+t)^{p'}$,
satisfies~\eqref{eqn:weak-frac1} with $A(t) =t^p\log(e+t)^{2p+\delta}$,
and $B(t)=t^{p'}\log(e+t)^{p'}$, and for any $b\in BMO$
the commutator $[b,I_\alpha]$ satisfies \eqref{eqn:weak-frac2}.
\end{theorem}

In both theorems the power of the logarithm on the function $A$ is
$2p+\delta$ instead of the conjectured $2p-1+\delta$; we believe that
this extra logarithm is not fundamental but rather is a consequence of
the proof.  The proof uses a two-weight inequality for the sharp
maximal function $M^\#$ which results in a loss of information.  The
proofs of Theorems~\ref{thm:factored-sio} and~\ref{thm:factored-frac}
can be adapted to prove Theorems~\ref{main} and~\ref{fraccomm} for
factored weights, but again in both cases we have to take
$A(t)=t^p\log(e+t)^{2p+\delta}$.  (Details are left to the interested
reader.)  As we noted above, this result for factored weights was one
motivation for initially conjecturing that Theorem~\ref{main} was
true. 

\subsection*{Organization}
The remainder of this paper is organized as follows.  In
Section~\ref{section:prelimaries} we gather a number of definitions
and results needed in our proofs.  In Section~\ref{section:local-mean}
we estimate the local mean oscillation of the commutator of a dyadic
singular integral, a key step in our proof of Theorem~\ref{main},
which we give in Section~\ref{section:proof-main-sio}.  In
Sections~\ref{section:proof-frac-onewt}
and~\ref{section:proof-fracccomm} we prove Theorems~\ref{sharpcom}
and~\ref{fraccomm} for commutators of fractional integrals.  In
Section~\ref{section:factored-proofs} we prove our weak type
inequalities for factored weights.  And finally, in
Section~\ref{section:sharp-examples} we construct the examples which
show that our results are sharp.  

Throughout this paper, all notation is standard or will be defined as
needed.  We will denote by $c$ a constant that generally depends only
on the dimension, the operator under consideration and the value of
$p$; the value of this constant, however, will often vary from line to
line.

\section{Preliminaries}
\label{section:prelimaries}
We start with some basic facts and notation.  By a weight we will mean
a measurable, non-negative function that is positive on a set of
positive measure. A pair of weights $(u,v)$
will always consist of non-negative, measurable functions such that:
$u>0$ on a set of positive measure, $u<\infty$ almost everywhere,
$v>0$ almost everywhere, and $v<\infty$ on a set of positive measure.
Given $p$, $1<p<\infty$, $p'$
will denote the dual exponent $p/(p-1)$.  For $1<p<\infty$ and a
weight $w$, $L^p(w)$ is the set of all measurable functions such that
$$\|f\|_{L^p(w)}=\left(\int_{\R^n} |f(x)|^p w(x)\, dx\right)^{1/p}<\infty.$$
When $w\equiv 1$, we write $L^p(\R^n)$.  

Hereafter, $Q$ will denote a cube.  Let $\D$ be the set of all dyadic
cubes in $\R^n$: i.e., cubes of the form $2^k(m+[0,1)^n)$ where $k\in
\Z$ and $m\in \Z^n$.  For $Q\in \D$, $\D(Q)$ is the set of all dyadic
subcubes of $Q$.  Given a dyadic cube $Q\in \D$ and an integer
$\tau\geq 0$, $Q^\tau$ will denote the unique dyadic cube containing $Q$
such that $|Q^\tau|=2^{\tau n}|Q|$.
 
Given a set $E$, we will use two different notions of an ``average" of
a function $f$ on the set $E$.  Let $a_f(E)$  denote the
mean value of $f$ on the set $E$:
 $$a_f(E)=\dashint_E f(x)\, dx=\frac{1}{|E|}\int_E f(x)\, dx.$$
Let $m_f(E)$ denote the median value of $f$ on $E$: the (possibly non-unique) number such that
$$\max\big(|\{x\in E: f(x)>m_f(E)\}|, |\{x\in E: f(x)<m_f(E)\}|\big)
\leq \frac{|E|}{2}.$$

\subsection{Dyadic operators} 
Below we will actually prove Theorems~\ref{main} and~\ref{fraccomm}
for dyadic singular and fractional integral operators.  Here we define
these operators and show how they can be used to approximate their
non-dyadic counterparts. 

\begin{definition} 
Given an integer $\tau\geq1$ we say $T^d$ is a dyadic singular integral of order $\tau$ if
$$T^d f(x)=\sum_{Q\in \D}\langle f, h_Q\rangle \cdot g_Q(x),$$
where $h_Q$ and $g_Q$ are functions that satisfy:
\begin{enumerate}[(i)]
\item $h_Q$ and $g_Q$ are supported on $Q$;
\item $h_Q$ and $g_Q$ are constant on $Q'\in \D(Q)$ with $|Q'|\leq 2^{-\tau n}|Q|$;
\item $\|h_Q\|_\infty, \|g_Q\|_\infty\leq |Q|^{-1/2}$;
\item $\dss \int_Q h_Q(x)\, dx=\int_Q g_Q(x)\, dx=0$.
\end{enumerate}
\end{definition}

Dyadic singular integrals are bounded on $L^2(\R^n)$ and of weak type
$(1,1)$.  The $L^2(\R^n)$ bounds follow from the Cotlar-Stein
lemma and the weak $(1,1)$ inequality follows from the usual
Calder\'on-Zygmund decomposition and the properties (ii) and (iv)
above.   (See~\cite{LPR2010}.)

The corresponding maximal truncated dyadic singular integral is
defined by
\begin{equation}\label{maxdyadic}
T^{d}_* f(x) =\sup_{l\in \Z} |T^{d}_lf(x)| 
\end{equation}
where
$$T_l^{d}f(x)=\sum_{{Q\in \D}\atop{|Q|\geq 2^{nl}}} \langle f, h_Q\rangle\cdot g_Q(x).$$
These operators also satisfy strong $(2,2)$ and weak $(1,1)$
inequalities (see \cite{HLRV-P2009}).

For $r>0$ and $\beta\in \R$,  let $r\D^\beta$ be the collection of cubes of the
form $r2^k(m+[\beta,\beta+1)^n)$, where $m\in \Z^n$.  Define the dyadic singular
integral operator of order $\tau$ adapted to $r\D^\beta$ by
$$T^{r,\beta}f(x)=\sum_{Q\in r\D^\beta} \langle f, h_Q\rangle\cdot g_Q(x),$$
where $h_Q$ and $g_Q$ satisfy properties (i), (ii), (iii), and (iv)
for cubes in $r\D^\beta$.  The classical singular integral operators
lie in the convex hull of the dyadic singular integral operators
adapted to $r\D^\beta$.  As a consequence we have the following
approximation theorem.  

\begin{theorem}[\cite{DV,Pet1, Pet2}] 
  Given $p$, $1<p<\infty$, suppose $T$ is the Hilbert transform, a
  Riesz transform, or the Ahlfors-Beurling operator.  Then there exists
  $\tau\geq 1$ (depending on $T$) and dyadic singular integral
  operators $\{T^{r,\beta}\}$ of order $\tau$ such that
\[  \|Tf\|_{L^p(\nu)}\leq 
c_\tau\sup_{{r>0}\atop{\beta \in \R^n}}\|T^{r,\beta}f\|_{L^p(\nu)}, \]
for all weights $\nu$ and functions $f$.
\end{theorem}

For example, the Hilbert transform can be approximated by dyadic
singular integrals of order $2$, the so called Haar shift operators.  Hence, to obtain a bound on the norm
of the Hilbert transform it suffices to bound the corresponding dyadic
singular integrals $T^{r,\beta}$ with a constant independent of $r$
and $\beta$.  Below we will prove estimates only for the standard
dyadic grid; it will be immediate that the same proofs yield bounds
for dyadic singular integral operators adapted to any grid
$r\D^\beta$.  

To apply our results to more general singular integral operators, we
would need to derive bounds on the dyadic singular integrals that were
polynomial in the order $\tau$.  However, the constants we get are
exponential in $\tau$; this is one of the obstacles that prevents us
from obtaining bounds for general singular integral operators as in
\cite{H}.  We will indicate the precise places where this occurs in
Remarks~\ref{exp1} and~\ref{exp2} below.  We do not know if our
methods can be modified to obtain a polynomial dependence on the order
$\tau$.

\medskip

The fractional integral operator is easier to approximate because its kernel is
positive and locally integrable.  Sawyer and Wheeden~\cite{SW}
introduced the dyadic fractional integral operator and proved it could
be used to approximate $I_\alpha$. 

\begin{definition} 
Given $\alpha$, $0<\al<n$, define the dyadic fractional integral
operator by
$$I_\al^d f(x)=\sum_{Q\in \D} |Q|^{\al/n} \dashint_Qf(y)\, dy \cdot
\chi_Q(x).$$
\end{definition}

To estimate $I_\al$ we only need
to average $I_\alpha^d$ over translations, $\tau_tf=f(\,\cdot -t)$.  

\begin{theorem}[\cite{SW}] Given $\alpha$, $0<\al<n$, and $p$,
  $1<p<\infty$, then
$$\|I_\al f\|_{L^p(\nu)}\leq c\sup_{\beta\in \R^n} \|\tau_\beta I^d_\al(\tau_{-\beta}f)\|_{L^p(\nu)}$$
for all weights $\nu$ and functions $f$. 
\end{theorem}

\subsection{Young functions and Orlicz spaces}

We follow the terminology and notation of~\cite{CMP4}.  A
function $\Phi$ is a Young function if $\Phi:[0,\infty)\ra [0,\infty)$ is
continuous, convex and strictly increasing, $\Phi(0)=0$ and
$\Phi(t)/t\ra \infty$ as $t\ra \infty$.  We will use the letters
$\Phi,\Psi,\ldots$ along with $A,B,\ldots$ to represent Young
functions.  The main examples we will be dealing with are
$\Phi(t)=t^r[\log(e+t)]^s$ for some $r\geq1$ and $s\in \R$.
(Hereafter we will write this more simply as $t^r\log(e+t)^s$.) 
Given a Young function $\Phi$,
the associate function $\bar \Phi$ is the Young function defined by
$$\bar{\Phi}(t)=\sup_{s>0}[st-\Phi(s)],\qquad t>0.$$
The functions $\Phi$ and $\bar{\Phi}$ satisfy 
$$t\leq \Phi^{-1}(t)\bar{\Phi}^{-1}(t)\leq 2t, \qquad t>0.$$ 
Given two Young functions $\Phi,\Psi$, we will use the notation
$\Phi(t)\approx \Psi(t)$ if there exists constants $c,C, t_0>0$ such
that for all $t\geq t_0$,
$$c\Phi(t)\leq \Psi(t)\leq C\Phi(t).$$

Given a cube $Q$, define the normalized Luxemburg norm of $f$ on $Q$
by 
$$\|f\|_{\Phi,Q}=\inf\left\{ \lambda>0: \dashint_Q \Phi\Big(\frac{|f(x)|}{\lambda}\Big)\, dx\leq 1\right\}.$$
 When $\Phi(t)=t^r$ for some $r>1$, then 
$$\|f\|_{\Phi,Q}=\left(\dashint_Q |f(x)|^r\,dx\right)^{1/r}\equiv \|f\|_{r,Q}.$$

There is a  generalized H\"older inequality for the Luxemburg norm.

\begin{lemma} \label{GenHolder}
If $\Phi,\Psi,$ and $\Theta$ are Young functions such that 
\[ \Phi^{-1}(t)\Psi^{-1}(t)\leq k\Theta^{-1}(t) \]
 for $t\geq t_0\geq 0$, then
$$\|fg\|_{\Theta,Q}\leq c \|f\|_{\Phi,Q}\|g\|_{\Psi,Q}.$$
In particular, for any Young function $\Phi$,
$$\dashint_Q |f(x)g(x)|\, dx \leq c \|f\|_{\Phi,Q}\|g\|_{\bar{\Phi},Q}.$$
\end{lemma}

Given a Young function $\Phi$ define the associated maximal operator by
$$M_\Phi f(x)=\sup_{Q\ni x} \|f\|_{\Phi,Q}.$$
There is also a dyadic version:
$$M^d_\Phi f(x)=\sup_{{Q\in \D}\atop{x\in Q}} \|f\|_{\Phi,Q}.$$
For each $\alpha$, $0<\al<n$,  define the associated fractional
maximal operators by
$$M_{\Phi,\al} f(x)=\sup_{Q\ni x} |Q|^{\al/n} \|f\|_{\Phi,Q},
\quad M^d_{\Phi,\al} f(x)=\sup_{{Q\in \D}\atop{x\in Q}}|Q|^{\al/n} \|f\|_{\Phi,Q}.$$
When $\Phi(t)=t\log(e+t)$ we will replace the subscript $\Phi$ with
$L\log L$; when $\Phi(t)\approx e^t$ we will replace the subscript
with $\exp L$. 

\medskip

As we noted in the Introduction, Young functions play an important
role in generalizing the $A_p$ condition to prove two-weight norm
inequalities.  Central to this are Young functions that satisfy the
following growth condition.  

\begin{definition} \label{defn:Bpcond}
For each $p$, $1<p<\infty$, a Young function $\Phi$ is said to belong to $B_p$ if for some $c>0$,
\begin{equation}\label{Bpcond} \int_c^\infty \frac{\Phi(t)}{t^p}\frac{dt}{t}<\infty.\end{equation}
\end{definition}

The next three results depend on the $B_p$ condition and will be used
in the proofs of our main results.  We start with a characterization
of $B_p$ in terms of the Orlicz maximal function due to
P\'erez \cite{P2}.

\begin{theorem} \label{Bpmax} 
For all $p$, $1<p<\infty$, $M_\Phi:L^p(\R^n)\ra L^p(\R^n)$ if and only if $\Phi\in B_p$.
\end{theorem}

We next give sufficient, $A_p$ bump conditions for two-weight
inequalities for the operators $M_\Phi$, $T^d$, and $T_*^d$.

\begin{theorem}[\cite{P2}] \label{LlogL}
Given $p$, $1<p<\infty$, let $\Phi,\Psi,$ and $\Theta$ be Young
functions such that $\Psi\in B_p$ and which satisfy
$\Phi^{-1}(t)\Psi^{-1}(t)\leq c\Theta^{-1}(t)$ for $t\geq t_0>0$.  If $(u,v)$ is a pair of weights such that 
$$\sup_Q \|u^{1/p}\|_{p,Q}\|v^{-1/p}\|_{\Phi,Q}<\infty,$$
then for every $f\in L^p(v)$,
$$\|M_\Theta f\|_{L^p(u)} \leq c\|f\|_{L^p(v)}.$$
\end{theorem}

\begin{theorem}[\cite{CMP3}] \label{2wdyad}
Let $T^d$ be a dyadic singular integral operator of order $\tau$, and
let $T^d_*$ be the associated maximal dyadic singular integral
operator.  Given $p$,  $1<p<\infty$, and Young functions  $\Phi,\,
\Psi$ 
such that $\bar{\Phi}\in B_{p'}$ and $\bar{\Psi}\in B_p$, if the pair
of weights $(u,v)$ satisfies
\begin{equation} \sup_Q\|u^{1/p}\|_{\Phi,Q}\|v^{-1/p}\|_{\Psi,Q}<\infty,\end{equation}
then for any $f\in L^p(v)$,
$$\|T^df\|_{L^p(u)}\leq c\|f\|_{L^p(v)}$$
and
$$\|T^d_*f\|_{L^p(u)}\leq c\|f\|_{L^p(v)}.$$
\end{theorem}

The next two norm inequalities will also be used below.  The first is
due to Yano; for a proof, see Zygmund~\cite{zygmund}.

\begin{theorem} \label{thm:yano}
Given a sub-linear operator $S$ that is bounded
  on $L^p(\R^n)$ for $1<p\leq p_0$, suppose that given any set
  $\Omega$ and $f$ such that $\supp(f)\subset \Omega$, 
$$\left(\dashint_\Omega|Sf(x)|^p\, dx\right)^{1/p} 
\leq \frac{c}{p-1}\left(\dashint_\Omega |f(x)|^p\,dx\right)^{1/p}.$$
then
$$\dashint_\Omega |Sf(x)|\, dx \leq c \|f\|_{L\log L,\Omega}.$$
\end{theorem}

It follows immediately from Marcinkiewicz interpolation that we can
take $S$ to be any operator that is bounded on $L^2(\R^n)$ and is weak
$(1,1)$.

The next result is a weak $(p,p)$ inequality for $M_{L\log
  L,\alpha}$.  It was proved in \cite[Proposition 5.16]{CMP4} for $\alpha=0$; the proof
for $\alpha>0$ is essentially the same.  For completeness we sketch
the details.

\begin{theorem} \label{thm:orlicz-weak}
Given $\alpha$, $0\leq\alpha<n$, and $p$, $1<p<n/\alpha$, if the pair
$(u,v)$ satisfies
\[ \sup_Q |Q|^{\al/n}\|u^{1/p}\|_{p,Q}\|v^{-1/p}\|_{B,Q} <\infty, \]
where $B(t)=t^{p'}\log(e+t)^{p'}$, then 
\[ u(\{ x\in \R^n : M_{L\log L,\al}f(x) > \lambda \}) 
\leq \frac{c}{\lambda^p}\int_{\R^n}|f(x)|^p v(x)\,dx. \]
\end{theorem}

\begin{proof}
By a variant of the Calder\'on-Zygmund decomposition for Orlicz
maximal operators (see
P\'erez~\cite{P2} and \cite{CF}), for each $\lambda>0$ there exists a
family of disjoint dyadic cubes $Q_j^\lambda$ and a constant
$\gamma>0$ such that $|Q_j^\lambda|^{\alpha/n}\|f\|_{L\log L,Q_j^\lambda}>\gamma\lambda$ and 
\[ \{ x\in \R^n : M_{L \log L,\al}f(x)>\lambda \} \subset \bigcup_j
3Q_j^\lambda. \]
If $\Phi(t)=t\log(e+t)$, then $B^{-1}(t)t^{1/p} \leq c \Phi^{-1}(t)$.
Therefore, by the generalized H\"older's inequality,
\begin{align*}
& u(\{ x\in \R^n : M_{L \log L,\al}f(x)>\lambda \} ) \\
& \qquad \leq \frac{c}{\lambda^p}\sum_j u(3Q_j^\lambda)|Q_j^\lambda|^{p\alpha/n}\|f\|_{L\log
  L, Q_j^\lambda}^p \\
& \qquad \leq \frac{c}{\lambda^p}\sum_j
|Q_j^\lambda|^{p\alpha/n}\|u^{1/p}\|_{p,3Q_j^\lambda}^p
\|v^{-1/p}\|_{B,3Q_j^\lambda}^p |3Q_j^\lambda|\|fv^{1/p}\|_{p,Q_j^\lambda}^p \\
& \qquad \leq c \int_{\R^n} |f(x)|^p v(x)\,dx.
\end{align*}
\end{proof}

\medskip

Finally, we give some special Young functions that will be used in our
proofs.  First, if $\Phi(t)=t\log(e+t)$, then a simple calculation
shows that $\bar{\Phi}(t) \approx e^t$.  We will use this to apply the
generalized H\"older's inequality.  

In Theorems~\ref{main} and~\ref{fraccomm} our hypotheses are stated in
terms of the Young functions 
\begin{align}
\label{A} A(t) & = t^p\log(e+t)^{2p-1+\delta}\\
\label{B} B(t) &=t^{p'}\log(e+t)^{2p'-1+\delta},
\end{align}
where $\delta>0$.  Closely related to these are the Young functions 
\begin{align}
\label{C} C(t)&=\frac{t^{p'}}{\log(e+t)^{1+(p'-1)\delta}}\\
\label{D} D(t)&=\frac{t^p}{\log(e+t)^{1+(p-1)\delta}}.
\end{align}

\begin{lemma} \label{Orliczmax} 
Fix $p$, $1<p<\infty$, and let $A,B,C,$ and $D$ be as in
\eqref{A},\eqref{B},\eqref{C} and \eqref{D}.
 Then $\bar{B},\, D\in B_p$ and $\bar{A},\,C\in B_{p'}$, and so 
\begin{equation*} 
M_{\bar{B}}, M_D:L^p(\R^n)\ra L^p(\R^n), \qquad M_{\bar{A}},
M_C:L^{p'}(\R^n)\ra L^{p'}(\R^n).  
\end{equation*}
Furthermore, if we let
 $\Phi(t)=t\log(e+t)$, then 
$$A^{-1}(t)C^{-1}(t)\leq c\Phi^{-1}(t) \quad \text{and}\quad B^{-1}(t)D^{-1}(t)\leq c\Phi^{-1}(t)$$
for $t\geq t_0>0$, and so for all $f,\,g$,
$$\|fg\|_{L\log L,Q}\leq c\|f\|_{A,Q}\|g\|_{C,Q}, \quad  \quad 
\|fg\|_{L\log L,Q}\leq c\|f\|_{B,Q}\|g\|_{D,Q}.$$
\end{lemma}

\begin{proof}
Straightforward calculations show that 
\begin{align*}
A^{-1}(t) &\approx \frac{t^{1/p}}{\log(e+t)^{1+1/p'+\delta/p}}\\
\bar{A}^{-1}(t)&\approx t^{1/p'}\cdot \log(e+t)^{1+1/p'+\delta/p}\\
\bar{A}(t)& \approx \frac{t^{p'}}{\log(e+t)^{p'+1+(p'-1)\delta}},
\end{align*}
and
\begin{align*}
C^{-1}(t)&\approx t^{1/p'}\cdot {\log(e+t)^{1/p'+\delta/p}}\\
\bar{C}^{-1}(t)&\approx \frac{t^{1/p'}}{\log(e+t)^{1/p'+\delta/p}}\\
\bar{C}(t)&\approx {t^{p}}\cdot{\log(e+t)^{p-1+\delta}}.
\end{align*}
Similar calculations hold for $B$ and $D$ (just exchanging the roles
of $p$ and $p'$).  The desired conclusions now follows from 
Definition~\ref{defn:Bpcond}, Lemma~\ref{GenHolder}, and Theorem~\ref{Bpmax}.
\end{proof}

\begin{remark} \label{rem:general-bump}
Since they are the principal examples, we have stated our main results in terms of Young functions $A$
and $B$ which are log bumps (i.e., of the form \eqref{A}, \eqref{B}).
However, 
we can actually prove somewhat more general results.  The key
properties we need are those given in Lemma~\ref{Orliczmax}.   Given a
Young function $A$, we will say that $C$ is its $L\log L$ associate if 
\[ A^{-1}(t)C^{-1}(t) \leq c\Phi^{-1}(t), \]
where $\Phi(t)=t\log(e+t)$.  Then we can restate the hypotheses of
Theorem~\ref{main} as follows: Given a Young function $A$ with $L\log
L$ associate $C\in B_{p'}$, and a Young function $B$ with $L\log L$
associate $D\in B_p$, if the pair $(u,v)$ satisfies~\eqref{twowcond},
then~\eqref{2weightcom} holds.  The hypotheses of
Theorem~\ref{fraccomm} may be reformulated similarly.  Details are
left to the interested reader.  Our proofs of the weak type results in
Theorems~\ref{thm:factored-sio} and~\ref{thm:factored-frac}, however,
only work for log bumps.
\end{remark}

\subsection{Bounded mean oscillation}
Let $BMO$ denote the space of functions of bounded mean oscillation:
functions $b$ such that 
$$\|b\|_{BMO}=\sup_Q \dashint_Q |b(x)-a_b(Q)|\, dx<\infty.$$
Below we will need that  $BMO$ functions  satisfy exponential
integrability conditions; this is a consequence of the John-Nirenberg
Theorem.  

\begin{theorem} \label{sharpJN}
Given $b\in BMO$, there exists a constant $c_n$ such that for every cube
$Q$, 
\begin{equation} \label{expBMO} 
\sup_Q \dashint_Q
\exp\Big(\frac{|b(x)-a_b(Q)|}{2^{n+2}\|b\|_{BMO}}\Big)\, dx \leq c_n.
\end{equation}
In particular, 
\begin{equation} \label{expBMO1}
\|b-a_b(Q)\|_{\exp L,Q} \leq c_n2^{n+2}\|b\|_{BMO}. 
\end{equation}

\end{theorem}

A proof of inequality \eqref{expBMO} is in Journ\'e~\cite{Jour}.  Inequality
\eqref{expBMO1} is an immediate consequence of~\eqref{expBMO} and the definition
of the Luxemburg norm.

\section{Estimates on the local mean oscillation of $[b,T^d]$}
\label{section:local-mean}

In this section we state a decomposition theorem due to
Lerner~\cite{Lern1} and make the estimate we need to apply it to
commutators of dyadic singular integrals.  We begin by recalling a few facts.
Given a cube $Q$ and $\lambda$, $0<\lambda<1$, define the local mean
oscillation of $f$ on $Q$ by
$$\omega_\lambda(f,Q)=\inf_{c\in \R}((f-c)\chi_{Q})^*(\lambda|Q|).$$
Define the dyadic local sharp maximal function on a fixed dyadic cube $Q$ by  
\begin{equation}\label{locsharp} 
M^{\sharp, d}_{\lambda, Q} f(x)=\sup_{{Q'\in \D(Q)}\atop{x\in Q'}} \omega_\lambda(f,Q')
\end{equation} 
By the properties of rearrangements, for all $p>0$, 
\begin{equation}\label{rearrang} 
(f\chi_Q)^*(\lambda |Q|)\leq
\lambda^{-1/p}\|f\|_{L^{p,\infty}(Q,dx/|Q|)}
\leq \lambda^{-1/p} \|f\|_{L^p(Q,dx/|Q|)}.
\end{equation}
Given a dyadic cube $Q$, $\hat{Q}$ will be its dyadic parent: the
unique dyadic cube of twice the side length of $Q$ that contains $Q$.
 
\begin{theorem}[\cite{Lern1}] \label{lerndec} 
Given a measurable function $f$ and a dyadic cube $Q$, for each $k\geq 1$ there exists a pairwise disjoint collection of cubes $\{Q^k_j\}\subset \D(Q)$ such that if $\Omega_k=\bigcup_j Q_j^k$:
\begin{enumerate}[{\rm(i)}]
\item $\Omega_{k+1}\subset \Omega_k$;
\item $|\Omega_{k+1}\cap Q_j^k|\leq \frac{1}{2}|Q_j^k|$;
\item for almost every $x\in Q$,
$$|f(x)-m_f(Q)|\leq c\Msharpquar f(x)+ c\sum_{j,k} \omega_{\frac{1}{2^{n+2}}}(f,\hat{Q}_j^k)\chi_{Q_j^k}(x).$$
\end{enumerate}
\end{theorem}

We make one observation which will be used heavily in what follows.
In general the sets $\{Q_j^k\}$ are only pairwise disjoint for a fixed
$k$.  However, if we define $E_j^k=Q_j^k\backslash \Omega_{k+1}$, then
the sets $\{E_j^k\}$ are pairwise disjoint for all $j,k$ and satisfy
$|E_j^k|\leq |Q^k_j|\leq 2|E_j^k|$.  

To apply Theorem \ref{lerndec} we need to estimate the local mean oscillation of $[b,T^d]$.   

\begin{lemma} \label{localos}
Suppose $T^d$ is a dyadic singular integral of
  order $\tau$, $Q$ is a dyadic cube and $0<\lambda\leq 1/2$.  Then
  there exists $c=c(n,\tau, \lambda)$ such that  for any $f$ and every
  $x\in Q$, 
\begin{equation}\label{meanosc} 
\omega_\lambda ([b,T^d]f, Q)\leq 
c\|b\|_{BMO} \big(\|f\|_{L\log L,Q^\tau}+\inf_{y\in Q} T_*^{d} f(y)\big),
\end{equation}
and
\begin{equation} \label{localshpest} 
\Msharp([b,T^d] f)(x)\leq c\|b\|_{BMO}\big(M_{L\log L}^d
f(x)+T_*^{d}f(x)\big). 
\end{equation}
\end{lemma}

\begin{proof} We will prove \eqref{meanosc}; \eqref{localshpest}
  follows at once from the definition of $\Msharp$.

Fix a dyadic cube $Q$ and decompose $T^d$ as
\begin{align*}
T^d f(x) & = \sum_{Q'\in \D} \langle f, h_{Q'}\rangle g_{Q'}(x) \\
& =\sum_{Q'\subseteq Q^\tau} \langle f, h_{Q'}\rangle g_{Q'}(x)
+\sum_{Q'\supset Q^\tau} \langle f, h_{Q'}\rangle g_{Q'}(x)\\ 
& =\Tin f(x)+\Tout f(x).
\end{align*}
The first term $\Tin$ is localized in the sense that $\Tin f(x)=\Tin
(f\chi_{Q^\tau})(x)$.  Furthermore, it is a dyadic singular
integral operator and so is bounded on $L^2(\R^n)$ and weak
$(1,1)$.  The second term $\Tout f(x)$ is constant on $Q$ since
 $Q'\supseteq Q^\tau$.  Thus for
$x\in Q$,
\begin{multline} \label{eqn:Tout-estimate}
|\Tout f(x)| =\Big|\sum_{Q'\supset Q^\tau} \langle f, h_{Q'}\rangle
g_{Q'}(x)\Big| \\
=\Big|\sum_{\stackrel{Q'\in \D}{|Q'|>2^{\tau n}|Q|}} \langle f,
h_{Q'}\rangle g_{Q'}(x)\Big|
\leq \inf_{y\in Q}  T_*^{d}f(y).
\end{multline}

To estimate the commutator, we rewrite it as
\begin{align*}
[b,T^d]f 
&= (b-a_b(Q^\tau))T^d f+T^d ((a_b(Q^\tau)-b)f)\\
&= (b-a_b(Q^\tau))\Tin f+(b-a_b(Q^\tau))\Tout f\\
&\qquad +\Tin ((a_b(Q^\tau)-b)f)+\Tout((a_b(Q^\tau)-b)f).
\end{align*}
The last term is constant on $Q$, so let
$c_{Q}=\Tout((a_b(Q^\tau)-b)f)(x)$ for some $x\in Q$.   Then we
can estimate the local oscillation of $[b,T_\tau^d]f$ by 
\begin{align*}
\omega_\lambda([b,T^d]f,Q)
& \leq  (([b,T^d]f-c_{Q})\chi_{Q})^*(\lambda|Q|) \\
&\leq [\Tin((b-a_b(Q^\tau))f)\chi_{Q}]^*\big(\frac{\lambda |Q|}{3}\big) \\
& \qquad +[(b-a_b(Q^\tau))(\Tin f)\chi_{Q}]^*\big(\frac{\lambda|Q|}{3}\big)\\
& \qquad +[(b-a_b(Q^\tau))(\Tout f)\chi_{Q}]^*\big(\frac{\lambda|Q|}{3}\big) \\
& = H_1 +H_2+H_3.
\end{align*}

We estimate each piece in turn.  By inequality~\eqref{rearrang}, the weak (1,1) boundedness of
$\Tin$ and the exponential integrability of $BMO$ functions (Theorem
\ref{sharpJN}),  we obtain
\begin{align*}
H_1 &\leq  c\lambda^{-1}\|\Tin((b-a_b(Q^\tau))f\chi_{Q^{\tau}})\|_{L^{1,\infty}(Q,dx/|Q|)}\\
&\leq  {c_\lambda} \dashint_{Q^\tau} |b-a_b(Q^\tau)|\, |f(x)|\, dx\\
&\leq  c_\lambda \|b-a_b(Q^\tau)\|_{\exp L, Q^\tau} \|f\|_{L\log L, Q^\tau}\\
&\leq  c_\lambda \|b\|_{BMO} \|f\|_{L\log L,Q^\tau}.\\
\end{align*}

To estimate  $H_2$ we use \eqref{rearrang} with $p=1/2$ and H\"older's
inequality to get
\begin{align*}
H_2 &\leq  \lambda^{-2}\left(\dashint_{Q} |(b-a_b(Q^\tau))\Tin f(x)|^{1/2}\, dx\right)^{2}\\
&\leq  c_\lambda \|b\|_{BMO}\dashint_{Q}|\Tin(f\chi_{Q^{\tau}})|\, dx\\
&\leq  c_\lambda \|b\|_{BMO}\|f\|_{L\log L,Q^\tau}.
\end{align*}
In the last inequality we used Yano's theorem
(Theorem~\ref{thm:yano}); this is possible
since $\Tin$ is bounded on $L^2$ and weak $(1,1)$.  

Finally we estimate $H_3$: by ~\eqref{rearrang} and~\eqref{eqn:Tout-estimate} we have that
$$H_3\leq \frac{c}{\lambda} \dashint_{Q} |\Tout
f(x)||b-a_b(Q^\tau)|\, dx
\leq c_\lambda \|b\|_{BMO}\inf_{y\in Q} T^{d}_* f(y).$$
\end{proof}

\begin{remark} \label{exp1}
  Inequalities \eqref{meanosc} and \eqref{localshpest} are the first
  of two points in which we pick up the exponential dependence on the
  parameter $\tau$.  In fact, if $c=c(n,\tau,\lambda)$ is the constant
  from \eqref{meanosc}, then careful examination shows that $c=c_\lambda
  2^{n\tau}$.
\end{remark}

\section{Proof of Theorem \ref{main}}
\label{section:proof-main-sio}

For the proof of Theorem \ref{main} we will need the following
estimate.  A similar inequality was proved in~\cite{CMP2}.

\begin{lemma} \label{doublemax} Given $p$, $1<p<\infty$, suppose the pair of
  weights $(u,v)$ satisfies
\begin{equation} \label{eqn:dmax}
\sup_Q\|u^{1/p}\|_{A,Q}\|v^{-1/p}\|_{B,Q}<\infty,
\end{equation}
where $A$ and $B$ are defined by \eqref{A} and \eqref{B}.  Then for
$f\in L^p(v)$ and $h\in L^{p'}(\R^n)$,
$$\int_{\R^n}M^df(x) M^d(u^{1/p}h)(x)\,dx\leq c \|f\|_{L^p(v)}\|h\|_{L^{p'}(\R^n)}.$$
\end{lemma}
\begin{proof} By a standard density argument we may assume $f,h$ are
  non-negative functions in $L^\infty_c$.  Set $a=4^n$ and let $w=u^{1/p}h$.  For each $j,k\in \Z$ define
$$\Omega_{j,k}=\{x: a^{k-j-1}<M^dw(x)\leq a^{k-j+1}\}\cap \{x: a^j < M^df(x)\leq a^{j+1}\};$$
then
$$\int_{\R^n}M^df(x) M^dw(x)\,dx\leq \sum_{j,k}\int_{\Omega_{j,k}}M^df(x)M^dw(x)\, dx.$$
For $l,m\in \Z$ let $\{P_r^l\}_r$ be the Calder\'on-Zygmund cubes of
$w$ at height $a^l$ and $\{Q^m_s\}_s$ be the Calder\'on-Zygmund cubes
of $f$ at height $a^m$ (see~\cite{CMP4,perez94}); then
$a_w(P_r^l)\approx a^l$, $a_f(Q_s^m)\approx a^m$, and 
$$\{x:M^dw(x)>a^{k-j-1}\}=\bigcup_{r} P_r^{k-j-1},\quad \{x:M^df(x)>a^j\}=\bigcup_s Q_s^{j}.$$
We then have that
$$\Omega_{j,k}\subseteq \bigcup_{r,s} P_r^{k-j-1}\cap Q_s^j.$$
Let $E_{j,k}^{r,s}=\Omega_{j,k}\cap( P_r^{k-j-1}\cap Q_s^j)$; if
$E_{j,k}^{r,s}\not=\varnothing$,
then either 
$$P_r^{k-j-1}\subseteq Q_s^j \quad \text{or} \quad  Q_s^j\subsetneq P_r^{k-j-1}.$$

Define $\Gamma_1,\Gamma_2\subset \Z^4$ by 
\begin{align*}
\Gamma_1 &=\{(j,k,r,s): P_r^{k-j-1}\subseteq Q_s^j \} \\
\Gamma_2 &= \{(j,k,r,s): Q_s^j\subsetneq P_r^{k-j-1}\}. 
\end{align*}
We can now estimate as follows:
\begin{align*}
\int_{\R^n}M^df(x) M^dw(x)\,dx 
&\leq \sum_{j,k}\int_{\Omega_{j,k}}M^df(x)M^dw(x)\, dx\\
&\leq  \sum_{j,k}\sum_{r,s}\int_{E_{j,k}^{r,s}} M^df(x)M^dw(x)\, dx \\
&\leq  \sum_{j,k}\sum_{r,s} a^{j+1}a^{k-j+1} |E_{j,k}^{r,s}|\\
&\leq \sum_{(j,k,r,s)\in \Gamma_1} a^{j+1}a^{k-j+1} |E_{j,k}^{r,s}|\\
& \qquad  + \sum_{(j,k,r,s)\in \Gamma_2} a^{j+1}a^{k-j+1} |E_{j,k}^{r,s}|\\
&= I_1+I_2.
\end{align*}
We first estimate $I_1$.   Let
$$\tilde{P}_r^l=P_r^l\backslash \{x:M^dw(x)>a^{l+1}\} 
\quad \text{and} \quad \tilde{Q}_s^m=Q_s^m\backslash
\{x:M^df(x)>a^{m+1}\};$$
then $|\tilde P_r^l|\geq \frac{1}{2}|P_r^l|, |\tilde Q_s^m|\geq
\frac{1}{2}|Q_s^m|$ and the families $\{\tilde P_r^l\}_{r,l}$ and
$\{\tilde Q_s^m\}_{s,m}$ are pairwise disjoint.  (See~\cite{perez94}.)
Further, for
$(j,k,r,s)\in \Gamma_1$ we have $P_r^{k-j-1}\subseteq Q_s^j$, and
$|E_{j,k}^{r,s}|\leq |P_r^{k-j-1}| \leq 2|\tilde
P_r^{k-j-1}|$. We  now estimate $I_1$:
\begin{align*}
I_1 
&\leq c\sum_{(j,k,r,s)\in \Gamma_1} 
\left(\dashint_{P_r^{k-j+1}} w(x)\ dx\right)\cdot 
\left(\dashint_{Q_s^{j}} f(x)\ dx\right)\cdot |E_{j,k}^{r,s}|\\
& \leq c   \sum_{j,s}\left(\dashint_{Q_s^{j}} f(x)\ dx\right)
\sum_{{k,r:}\atop{(j,k,r,s)\in \Gamma_1}}  
\dashint_{P_r^{k-j+1}} w(x)\ dx \cdot |\tilde P_r^{k-j+1}| \\
&\leq c  \sum_{j,s}\left(\dashint_{Q_s^{j}} f(x)\ dx\right)
\sum_{{k,r:}\atop{(j,k,r,s)\in \Gamma_1}}  \int_{\tilde{P}_r^{k-j+1}}M^d(\chi_{Q_s^j} w)(x)\ dx \\
&\leq c\sum_{j,s}\left(\dashint_{Q_s^{j}} f(x)\ dx\right) \cdot 
\left(\dashint_{Q_s^j}M^d(\chi_{Q_s^j} w)(x)\ dx\right)\cdot |\tilde{Q}_s^j| \\
&\leq c \sum_{j,s}\left(\dashint_{Q_s^{j}} f(x)\ dx\right) \cdot 
\|w\|_{L\log L,Q_s^j}\cdot |\tilde{Q}_s^j| \\
&\leq  c\sum_{j,s} \|fv^{1/p}\|_{\bar{B},Q_s^j}
\|v^{-1/p}\|_{B,Q_s^j} 
\|h\|_{C,Q_s^j}\|u^{1/p}\|_{A,Q_s^j} \cdot |\tilde{Q}_s^j|;
\end{align*}
the Young function $C$ is as in equation \eqref{C} and  we have used the
generalized H\"older inequality (Lemma \ref{GenHolder}) and Yano's
theorem (Theorem~\ref{thm:yano}) in the second to last inequality.  
By Lemma \ref{Orliczmax},
$M_{\bar{B}}$ and $M_C$ are bounded on $L^p(\R^n)$ and $L^{p'}(\R^n)$
respectively.  Hence,  by~\eqref{eqn:dmax} and H\"older's inequality with respect to the summation,
\begin{align*}
I_1
&\leq c \left(\sum_{j,s} \|fv^{1/p}\|_{\bar{B},Q_s^j}^p\cdot
  |\tilde{Q}_s^j|\right)^{1/p}
\left(\sum_{j,s} \|h\|_{C,Q_s^j}^{p'} \cdot |\tilde{Q}_s^j|\right)^{1/p'}\\
&\leq  c \left(\sum_{j,s}\int_{\tilde{Q}_s^j}
  M_{\bar{B}}(fv^{1/p})(x)^p\, dx\right)^{1/p}
\left(\sum_{j,s} \int_{\tilde{Q}_s^j} M_{C}h(x)^{p'}\, dx\right)^{1/p'}\\
&\leq  c\left(\int_{\R^n} M_{\bar{B}}(fv^{1/p})(x)^p\, dx\right)^{1/p}
\left(\int_{\R^n} M_{C}h(x)^{p'}\, dx\right)^{1/p'}\\
&\leq  c\|f\|_{L^p(v)}\|h\|_{L^{p'}(\R^n)}.
\end{align*}

The estimate for
$I_2$ is similar.  Since $E_{j,k}^{r,s}\subseteq
Q_s^j\subsetneq P_r^{k-j-1}$ for $(j,k,r,s)\in \Gamma_2$ we have that
\begin{align*}
I_2  &\leq c\sum_{(j,k,r,s)\in \Gamma_2}\left(\dashint_{P_r^{k-j+1}}
  w(x)\ dx\right)\cdot 
\left(\dashint_{Q_s^{j}} f(x)\ dx\right)\cdot |E_{j,k}^{r,s}|\\
&\leq  c\sum_{r,l} \dashint_{P_r^l} w(x)\ dx\sum_{{(j,k,r,s)\in
    \Gamma_2}\atop{k-j-1=l}}
\left(\dashint_{Q_s^{j}} f(x)\ dx\right)\cdot |E_{j,k}^{r,s}|\\
&\leq  c\sum_{r,l} \dashint_{P_r^l} w(x)\ dx\sum_{{(j,k,r,s)\in
    \Gamma_2}\atop{k-j-1=l}}
\left(\dashint_{Q_s^{j}} f(x)\ dx\right)\cdot |\tilde{Q}_s^j|\\
&\leq  c\sum_{r,l}\left( \dashint_{P_r^l} w(x)\ dx\right)\cdot
\left(\dashint_{P_r^l} M^d(\chi_{P_r^l}f)(x)\,dx\right)\cdot |\tilde P_r^l|\\
&\leq  c\sum_{r,l}\left( \dashint_{P_r^l} u(x)^{1/p}h(x)\
  dx\right)\cdot 
\|f\|_{L\log L,P_r^l}\cdot |\tilde P_r^l|\\
&\leq  c\sum_{r,l}\|u^{1/p}\|_{A,P_r^l}\|h\|_{\bar{A},P_r^l}
\|fv^{1/p}\|_{D,P_r^l}\|v^{-1/p}\|_{B,P_r^l}\cdot |\tilde P_r^l|\\
&\leq  c\|f\|_{L^p(v)}\|h\|_{L^{p'}(\R^n)};
\end{align*}
$D$ is as in
\eqref{D} and we have  once again used~\eqref{eqn:dmax}, Yano's theorem, and Lemma
\ref{Orliczmax} for the boundedness of $M_{\bar{A}}$ and $M_D$ on
$L^{p'}(\R^n)$ and $L^p(\R^n)$.
\end{proof}

\begin{proof}[Proof of Theorem \ref{main}] 
The first part of our argument is similar to one found in \cite[Theorems~5.1,~5.2]{CMP3}.  Fix $f$; by a standard approximation
argument we may assume without loss of generality that $f\in
L^\infty_c$.   Let $\R^n_j$, $1\le j\le 2^n$, denote the the
$n$-dimensional quadrants in $\R^n$: i.e., the sets
$\R^{\pm}\times \R^{\pm}\times\cdots\times \R^{\pm}$ where $\R^+ =
[0,\infty)$ and $\R^-=(-\infty,0)$.  For each $j$, $1\le j\le 2^n$, and for each $N>0$ let $Q_{N,j}$ be
the dyadic cube adjacent to the origin of side length $2^N$ that is
contained in $\R^n_j$.   Since $T^d$ is weak $(1,1)$ and strong
$(2,2)$, by interpolation and duality it is bounded on $L^p(\R^n)$,
$1<p<\infty$.  Therefore, since $|m_f(Q)|\leq  (f\chi_Q)^*(|Q|/2)$ (see~\cite{Lern1}), by
inequality~\eqref{rearrang}, $m_{[b,T^d]f}(Q_{N,j})\rightarrow 0$ as
$N\rightarrow \infty$.  Therefore, by Fatou's lemma and Minkowski's
inequality,
\begin{multline*}
 \|[b,T^d] f\|_{L^p(u)}  \\\leq \liminf_{N\rightarrow \infty}
\sum_{j=1}^{2^n} \left(\int_{Q_{N,j}} |[b,T^d]f(x) - m_{[b,T^d]f}(Q_{N,j})|^pu(x)\,dx\right)^{1/p}.
\end{multline*}
Hence, it will suffice to prove that each term in the sum on the
right is bounded by $c\|f\|_{L^p(v)}$ where $c$ is independent of
$N$.  Further, by duality, it will suffice to show that for any $h\in
L^{p'}$, $\|h\|_{p'}=1$, 
\[ \int_{Q_{N,j}} |[b,T^d]f(x) -
m_{[b,T^d]f}(Q_{N,j})|u(x)^{1/p}h(x)\,dx \leq c\|f\|_{L^p(v)}. \]

Fix $j$ and let $Q_N=Q_{N,j}$.  By Theorem \ref{lerndec} and Lemma
\ref{localos} we have the following pointwise estimate:
\begin{align*}
\lefteqn{|[b,T^d] f(x)-m_{[b,T^d]f}(Q_N)|}\\
& \quad \leq  c\M^{\sharp,d}_{\frac14,Q_N}([b,T^d]f)(x) +c\sum_{j,k} \omega_{\frac{1}{2^{n+2}}}([b,T^d]f, \hat{Q}^k_j)\chi_{Q_j^k}(x)\\
& \quad \leq c\|b\|_{BMO} \big(M^d_{L\log L}f(x)+T_*^{d}f(x)\\
& \quad \qquad +\sum_{j,k} \|f\|_{L\log L, P_j^k}\chi_{Q_j^k}(x)+
\sum_{j,k}\inf_{y\in Q_j^k} T_*^df(y)\chi_{Q^k_j}(x)\big)\\
&\quad = c\|b\|_{BMO} (M^d_{L\log L}f(x)+T_*^{d}f(x)+F(x)+G(x)),
\end{align*}
where
$P_j^k=(\hat{Q}_j^k)^\tau$.  Fix $h\in L^{p'}(\R^n)$, $\|h\|_{p'}=1$; then we have
\begin{align*}
 \lefteqn{\int_{Q_N} |[b,T^d] f(x)-m_{[b,T^d]f}(Q_N)| u(x)^{1/p}h(x)\, dx} \\
&\quad \leq c\|b\|_{BMO}\left( \int_{Q_N} M^d_{L\log L} f(x)u(x)^{1/p}h(x)\, dx\right. \\
& \quad \quad +\int_{Q_N} T_*^{d}f(x)u(x)^{1/p}h(x)\, dx+\int_{Q_N}F(x)u(x)^{1/p}h(x)\, dx\\
& \quad \quad \left. +\int_{Q_N}G(x)u(x)^{1/p}h(x)\, dx\right)\\
&\quad = c\|b\|_{BMO}(J_1+J_2+J_3+J_4).
\end{align*}
We first note that $J_1$ and $J_2$ are bounded by $\|f\|_{L^p(v)}$, since
the pair $(u,v)$ satisfies the conditions for the two-weight
norm inequalities for the operators $M_{L\log L}^d$ and $T^{d}_*$.  More
precisely, by H\"older's inequality and Theorem \ref{LlogL} we have
that 
\[  J_1\leq \|M_{L\log L}^d f\|_{L^p(u)} \|h\|_{L^{p'}(\R^n)} \leq
c\|f\|_{L^{p}(v)}. \]
Similarly, by Theorem \ref{2wdyad}, 
\[ J_2\leq \|T^{d}_* f\|_{L^p(u)}\|h\|_{L^{p'}(\R^n)}\leq
c\|f\|_{L^p(v)}. \]

Let $E_j^k=Q_j^k\backslash \Omega_{k+1}$ so that the sets $\{E_j^k\}$ are pairwise disjoint and satsify $|E_j^k|\approx |Q_j^k|$ (see the comment following Theorem \ref{lerndec}).  We now estimate $J_3$:
\begin{align*}
J_3
&= \sum_{j,k} \|f\|_{L\log L, P_j^k} \cdot \dashint_{Q_j^k} u(x)^{1/p} h(x)\, dx \cdot |Q_j^k| \\
&\leq  c\sum_{j,k} \|f\|_{L\log L, P_j^k} \cdot \dashint_{Q_j^k} u(x)^{1/p} h(x)\, dx \cdot |E_j^k| \\
&\leq  c \sum_{j,k} \| f v^{1/p}\|_{D, P_j^k} \|v^{-1/p}\|_{B,P_j^k} 
\|h\|_{\bar{A},Q_j^k} \|u^{1/p}\|_{A,Q_j^k} |E_j^k|,
\end{align*}
where $D$ is from \eqref{D}.  By Lemma \ref{Orliczmax}
$M_D:L^p(\R^n)\ra L^p(\R^n)$ and $M_{\bar{A}}:L^{p'}(\R^n)\ra
L^{p'}(\R^n)$.  Hence, by~\eqref{twowcond},
\begin{align*}
J_3 &\leq c\sum_{j,k} \| f v^{1/p}\|_{D, P_j^k} \|v^{-1/p}\|_{B,P_j^k} 
\|h\|_{\bar{A},Q_j^k} \|u^{1/p}\|_{A,Q_j^k} |E_j^k| \\
&\leq c\left(\sum_{j,k}\| f v^{1/p}\|_{D, P_j^k}^p\cdot
  |E_j^k|\right)^{1/p}  
\left(\sum_{j,k}\| h\|_{\bar{A}, Q_j^k}^{p'}\cdot |E_j^k|\right)^{1/p'}\\
 &\leq c \left(\sum_{j,k}\int_{E_j^k}M_{D}(f v^{1/p})(x)^p\,
   dx\right)^{1/p}  
\left(\sum_{j,k}\int_{E_j^k}M_{\bar{A}}h(x)^{p'}\, dx\right)^{1/p'}\\
&\leq  c \left(\int_{\R^n} M_D(fv^{1/p})(x)^p\, dx\right)^{1/p}
\left(\int_{\R^n} M_{\bar{A}}h(x)^{p'}\, dx\right)^{1/p'}\\
&\leq  c\| f\|_{L^p(v)}.
\end{align*}

Finally we estimate $J_4$.  We have 
\begin{multline*}
J_4\leq \sum_{j,k} \big(\inf_{y\in Q_j^k} T_*^{d} f(y) \big) \cdot \dashint_{Q_j^k} u(x)^{1/p}h(x)\,dx \cdot |E_j^k|\\
\leq c\int_{\R^n} T_*^{d} f(x) M^d(u^{1/p}h)(x)\, dx.
\end{multline*}
To estimate the right-hand term, we apply a reduction argument very
similar to the one given above to show that it will suffice to prove
\begin{multline} \label{eqn:max-sio}
 \int_{Q_N} \sup_{l\in \Z}|T_l^{d}f(x)-m_{T_l^{d}f}(Q_N)|
M^d(u^{1/p}h)(x)\, dx \\
\leq c\|f\|_{L^p(v)}\|h\|_{L^{p'}(\R^n)}.
\end{multline}
(For the details of this reduction for maximal dyadic singular
integrals, see \cite[Theorem~6.1]{CMP3}.)

To prove~\eqref{eqn:max-sio} we again use the Lerner decomposition
argument.  As was shown in~\cite{CMP3}, we have that
$$\sup_{l\in \Z} \omega_\lambda(T_l^{d}f, Q)\leq c\dashint_{Q^\tau} |f(x)|\, dx,$$
and so
$$\sup_{l\in \Z} M_{\lambda, Q}^{\sharp,d}(T_l^{d}f)(x)\leq c M^df(x).$$
Therefore, by Theorem~\ref{lerndec},
\begin{align*}
& \int_{Q_N} \sup_{l\in \Z}|T_l^{d}f(x)-m_{T_l^{d}f}(Q_N)| M^d(u^{1/p}h)(x)\, dx\\
&\qquad \leq  c\int_{\R^n} M^df(x)M^d(u^{1/p}h)(x)\, dx \\
& \qquad \qquad +\sum_{j,k} \dashint_{P_j^k} |f(x)|\, dx\cdot 
\dashint_{Q_j^k} M^d(u^{1/p}h)(x)\, dx \cdot |E_j^k|\\
& \qquad = c\ (J_5+J_6).
\end{align*}
(Note that the families of cubes $\{Q_j^k\}$ and $\{P_j^k\}=\{(\hat{Q}_j^k)^\tau\}$ are different from the families in the first part of the proof.)

To estimate  $J_5$ we use Lemma \ref{doublemax} to get
$$J_5\leq \|f\|_{L^p(v)}\|h\|_{L^{p'}(\R^n)}.$$
To estimate $J_6$, we argue as follows:
\begin{align*}
J_6
&=  \sum_{j,k} \dashint_{P_j^k} |f(x)|\, dx\cdot 
\dashint_{Q_j^k} M^d(u^{1/p}h)(x)\, dx \cdot |E_j^k|\\
&\leq  \sum_{j,k} \dashint_{P_j^k} |f(x)|\, dx\cdot 
\dashint_{Q_j^k} M^d(u^{1/p}h\chi_{Q_j^k})(x)\, dx \cdot |E_j^k|\\
& \qquad + \sum_{j,k} \dashint_{P_j^k} |f(x)|\, dx\cdot 
\dashint_{Q_j^k} M^d(u^{1/p}h\chi_{\R^n\backslash Q_j^k})(x)\, dx \cdot |E_j^k|\\
&= J_7+J_8.
\end{align*}

For $J_7$ we argue as we did in the estimate for $J_3$ above to get
\begin{align*}
J_7 
&= \sum_{j,k} \dashint_{P_j^k} |f(x)|\, dx\cdot 
\dashint_{Q_j^k} M^d(u^{1/p}h\chi_{Q_j^k})(x)\, dx \cdot |E_j^k| \\
&\leq  c\sum_{j,k} \dashint_{P_j^k} |f(x)|\, dx\cdot 
\|u^{1/p}h\|_{L\log L,Q_j^k} \cdot |E_j^k| \\
& \leq  c\sum_{j,k} \|f v^{1/p}\|_{\bar{B},P_j^k}
\|v^{-1/p}\|_{B,P_j^k} 
\|u^{1/p}\|_{A,Q_j^k} \|h\|_{C,Q_j^k} |E_j^k|\\
&\leq  c\|f\|_{L^p(v)}.
\end{align*}

To estimate $J_8$, first note that $ M^d(u^{1/p}h\chi_{\R^n\backslash
  Q_j^k})$ is constant on $Q_j^k$: for $x\in Q_j^k$,
$$M^d(u^{1/p}h\chi_{\R^n\backslash Q_j^k})(x)=
\sup_{{Q\in \D}\atop{Q\supsetneq Q_j^k}}\frac{1}{|Q|}\int_{Q\backslash Q_j^k}  |u(y)^{1/p}h(y)|\, dy.$$
Hence,
\begin{align*}
J_8 
& = \sum_{j,k} \dashint_{P_j^k} |f(x)|\, dx \cdot 
\Big( \inf_{y\in Q_j^k}M^d(u^{1/p}h\chi_{\R^n\backslash Q_j^k})(y)\Big)\cdot |E_j^k|\\
&\leq  c\sum_{j,k} \int_{E_j^k} M^df(x)M^d(u^{1/p}h)(x)\, dx\\
&\leq  c\int_{\R^n} M^df(x)M^d(u^{1/p}h)(x)\, dx\\
&\leq  c\|f\|_{L^p(v)},
\end{align*}
where the last inequality follows from Lemma \ref{doublemax}.
\end{proof}

\begin{remark} \label{exp2}
The second point at which we pick up exponential dependence on $\tau$
is in the estimates of $J_3$ and $J_7$ above.  In order to
use~\eqref{eqn:sharp-sio1} to estimate
\[ \|u^{1/p}\|_{A,Q_j^k}\|v^{-1/p}\|_{B,P_j^k}, \]
we have to replace $Q_j^k$ by $P_j^k$ in the first term.  Since
$|P_j^k|=2^{n(\tau+1)}|Q_j^k|$, by the homogeneity of the norm we can
do so at the cost of a constant $2^{n(\tau+1)/p}$
(see~\cite[Section~5.2]{CMP4}).  
\end{remark}

\section{Proof of Theorem~\ref{sharpcom}} 
\label{section:proof-frac-onewt}

Our proof is similar to that for commutators of singular integrals
in~\cite{CPP}.  By the sharp, off-diagonal extrapolation theorem in
\cite{LMPT}, it suffices to prove~\eqref{sharpbound} in the particular
case
\[ \frac{2}{p} = 1+ \frac{\alpha}{n}.  \]
It follows at once that in this case, $q=p'$.   

By a standard approximation argument, we may assume $f\in
C_c^\infty(\R^n)$.  Given this assumption, we can represent the
commutator using the Cauchy integral formula: for all $\epsilon>0$,
\begin{equation*}
[b,I_\al]f(x)=\frac{1}{2\pi i}\int_{|\zeta|=\ep}\frac{e^{\zeta
    b}I_\al(e^{-\zeta b}f)(x)}{\zeta^2} d\zeta.
\end{equation*}
(See~\cite{ABKP,CRW}.)  Fix $w\in A_{p,p'}$; then  by Minkowski's integral
inequality we have that
\begin{equation} \label{eqn:Cauchy}
 \|[b,I_\alpha]f\|_{L^q(w^q)} \leq \frac{1}{2\pi}
\int_{|\zeta|=\epsilon} |\zeta|^{-2} \|e^{\zeta b}I_\alpha(e^{-\zeta
  b}f)\|_{L^{p'}(w^{p'})} \,d|\zeta|. 
\end{equation}

We now estimate 
$$\|e^{\zeta b}I_\al(e^{-\zeta b}f)\|_{L^{p'}(w^{p'})}=
\|I_\al(e^{-\zeta b}f)\|_{L^{p'}(e^{p'Re\, \zeta\, b}w^{p'})}.$$  
Since $q=p'$, it
follows from the definitions that since $w\in A_{p,p'}$, then $w^{p'} \in
A_2$ and 
\[ [w]_{A_{p,p'}}=[w^{p'}]_{A_2}=[w^{-p'}]_{A_2}. \]  
Therefore, both $w^{p'}$ and $w^{-p'}$ satisfy the
reverse H\"older inequality.  In particular, by the sharp 
reverse H\"older inequality in \cite[Lemma~8.1]{P3}  (see also
\cite[Lemma~2.3]{CPP}), for every cube $Q$,
\[ \left(\dashint_Q w(x)^{\pm p'r}\,dx\right)^{1/r} \leq 2\dashint_Q
w(x)^{\pm p'}\,dx, \]
where 
\[ r = 1 + \frac{1}{2^{n+5}[w]_{A_{p,p'}}}. \]
If we first apply H\"older's inequality  with this exponent to the two integrals in the
definition of $A_{p,p'}$ and then apply the reverse H\"older
inequality, we get that
$$[e^{Re\, \zeta\, b}w]_{A_{p,p'}}\leq
[w^r]_{A_{p,p'}}^{1/r}[e^{r'Re\, \zeta\, b}]_{A_{p,p'}}^{1/r'}
\leq  4[w]_{A_{p,p'}} [e^{r'Re\, \zeta\, b}]_{A_{p,p'}}^{1/r'}.$$

Fix $\epsilon>0$ such that 
\begin{equation*}\label{ep} 
\ep=\frac{2^{-(n+2)}}{r'\|b\|_{BMO}} 
\approx  ([w]_{A_{p,p'}}\|b\|_{BMO})^{-1}.
\end{equation*}
Then it follows from Theorem~\ref{sharpJN} (see \cite[Lemma~2.2]{CPP})
that $e^{r'Re\, \zeta\, b}\in A_{p,p'}$ and $ [e^{r'Re\, \zeta\,
  b}]_{A_{p,p'}} \leq c_n$, (where $c_n$ is the constant in
Theorem~\ref{sharpJN}). 

Hence, if we combine this estimate with the sharp inequality for
the fractional integral operator \eqref{sharpfrac}, we get
\begin{multline*}
\|I_\al(e^{-\zeta b}f)\|_{L^{p'}(e^{p'Re\, \zeta\, b}w^{p'})} \\
\leq c [e^{Re\, \zeta \, b}w]_{A_{p,p'}}^{1-\frac{\al}{n}} 
\|f e^{-\zeta b}\|_{L^{p}(w^{p}e^{pRe\, \zeta\, b})}
\leq  c[w]_{A_{p,p'}}^{1-\frac{\al}{n}}\|f\|_{L^{p}(w^p)}.
\end{multline*}
This inequality together with \eqref{eqn:Cauchy} then yields
$$\|[b,I_\al]f\|_{L^{p'}(w^{p'})}
\leq c{\ep^{-1}}[w]_{A_{p,p'}}^{1-\frac{\al}{n}}\|f\|_{L^{p}(w^p)}
=c\|b\|_{BMO}[w]_{A_{p,p'}}^{2-\frac{\al}{n}}\|f\|_{L^{p}(w^p)}.$$
This completes the proof.

\section{Proof of Theorem \ref{fraccomm}}
\label{section:proof-fracccomm}

By duality, it will suffice
to prove that for all $f\in L^p(v)$ and all $h\in L^{q'}(\R^n)$,
$\|h\|_{q'}=1$, 
\[ \int_{\R^n} |[b,I^d_\al]f(x)| h(x)u(x)^{1/q}\, dx \leq
c\|f\|_{L^p(v)}. \]
By a standard approximation argument we may assume $f,\,h \in
L_c^\infty$.   Further, since $I_\alpha^d$ is a positive operator, we
may assume $f$ and $h$ are non-negative. 

Fix $f$ and $h$.  Then
\begin{align*}
& \int_{\R^n} |[b,I^d_\al]f(x)| h(x)u(x)^{1/q}\, dx\\
 & \qquad \leq \sum_{Q\in \D} \frac{|Q|^{\al/n}}{|Q|}\int_Q\int_Q |b(x)-b(y)|f(y)h(x)u(x)^{1/q}\, dy\, dx \\
& \qquad \leq  \sum_{Q\in \D} {|Q|^{\al/n}}\dashint_Q
|b(x)-a_b(Q)|h(x)u(x)^{1/q}\, dx\cdot 
\int_Q f(y)\, dy \\
& \qquad \qquad + \sum_{Q\in \D} {|Q|^{\al/n}}\dashint_Q |b(y)-a_b(Q)|f(y)\,
dy\cdot 
\int_Q h(x)u(x)^{1/q}\, dx\\
& \qquad = K_1+K_2.
\end{align*}
We will estimate $K_1$; the estimate for $K_2$ is gotten in the same way,
exchanging the roles of $f$ and $u^{1/p}h$.  By H\"older's inequality
(Lemma~\ref{GenHolder}) and the exponential integrability of $BMO$
functions (Theorem~\ref{sharpJN}), we have that 
\begin{multline*}
\dashint_Q|b(x)-a_b(Q)|h(x)u(x)^{1/q}\, dx  \\
\leq c\|b-a_b(Q)\|_{\exp L, Q}\|hu^{1/q}\|_{L\log L,Q}
\leq c\|b\|_{BMO}\|hu^{1/q}\|_{L\log L, Q}.
\end{multline*}
Hence, 
\[  K_1 \leq c\|b\|_{BMO} \sum_{Q\in \D}
|Q|^{\al/n}\|hu^{1/q}\|_{L\log L,Q}\cdot \int_Q f(y)\, dy.  \] 
By an argument in \cite{CMP1} (see also P\'erez~\cite{perez94} ) we
may replace the sum over all dyadic cubes with the sum over the
Calder\'on-Zygmund cubes of $f$.  More precisely, for each $k\in \Z$,
let $\{Q_j^k\}$ be the set of disjoint maximal dyadic cubes such that
$$\Omega_k =\{x:M^df(x)>a^k\}=\bigcup_j Q_j^k,$$
where $a=4^n$.  Let $E_j^k=Q_j^k\backslash \Omega_{k+1}$;  then the
sets $E_j^k$ are pairwise disjoint for all $j$ and $k$, and
$|E_j^k|\geq \frac{1}{2}|Q_j^k|$. Then
\begin{multline*}
\sum_{Q\in \D}
|Q|^{\al/n}\|hu^{1/q}\|_{L\log L,Q}\cdot \int_Q f(y)\, dy  \\\leq c
\sum_{j,k}
|Q_j^k|^{\al/n}\|hu^{1/q}\|_{L\log L,Q_j^k}\cdot \int_{Q_j^k} f(y)\,
dy.
\end{multline*}
%
%
Define the Young function $C_q$ by
\[ C_q(t)=\frac{t^{q'}}{\log(e+t)^{1+(q'-1)\delta}}.  \]
By the same argument as in the proof of Lemma~\ref{Orliczmax}, we have
that $C_q\in B_{q'}$ and $A_q^{-1}(t)C_q^{-1}(t)\leq c\Phi^{-1}(t)$,
where $\Phi(t)=t\log(e+t)$.   Further, by the same lemma, $\bar{B}\in
B_p$.  Therefore, by the generalized
H\"older's inequality (Lemma~\ref{GenHolder}),
\eqref{twowcondfrac}, and  H\"older's inequality with respect to
the summation,
\begin{align*}
K_1 &  \leq  c\|b\|_{BMO} \\
& \qquad \times \sum_{j,k} |Q_j^k|^{\al/n}
\|h\|_{C,Q_j^k}\|u^{1/q}\|_{A_q,Q_j^k} \|fv^{1/p}\|_{\bar{B},Q_j^k} \|v^{-1/p}\|_{B,Q_j^k} \cdot |Q_j^k| \\
&  \leq  c\|b\|_{BMO}\sum_{j,k}\|h\|_{C_q,Q_j^k}
\|fv^{1/p}\|_{\bar{B},Q_j^k}  \cdot |Q_j^k|^{1/q'+1/p}\\
&  \leq
c\|b\|_{BMO}\left(\sum_{j,k}\|fv^{1/p}\|_{\bar{B},Q_j^k}^p
|E_j^k|\right)^{1/p}  \cdot
\left(\sum_{j,k}\|h\|_{C_q,Q_j^k}^{p'}|E_j^k|^{p'/q'}\right)^{1/p'}.
\end{align*}
Since $p\leq q$, $p'/q'\geq 1$.  Therefore, by convexity and Theorem~\ref{Bpmax},
\begin{align*}
K_1 &  \leq
c\|b\|_{BMO}\left(\sum_{j,k}\|fv^{1/p}\|_{\bar{B},Q_j^k}^p
|E_j^k|\right)^{1/p}  \cdot \left(\sum_{j,k}\|h\|_{C_q,Q_j^k}^{q'}|E_j^k|\right)^{1/q'} \\
&  \leq  c\|b\|_{BMO}\left(\int_{\R^n}
  M_{\bar{B}}(fv^{1/p})(x)^p\, dx\right)^{1/p}  
\cdot \left(\int_{\R^n}M_{C_q}h(x)^{q'}\, dx\right)^{1/q'} \\
&  \leq  c\|b\|_{BMO} \left(\int_{\R^n} |f(x)|^pv(x)\,
  dx\right)^{1/p}  
\cdot \left(\int_{\R^n}|h(x)|^{q'}\, dx\right)^{1/q'} \\
&   = c\|b\|_{BMO}\|f\|_{L^p(v)}.
\end{align*}
This completes the proof.

\section{Proofs of Theorems~\ref{thm:factored-sio} and \ref{thm:factored-frac}}
\label{section:factored-proofs}

We need two lemmas, both taken from~\cite{CMP4}.

\begin{lemma}[\mbox{\cite[Theorems~6.4,~6.16]{CMP4}}]\label{lemma:reverse-factorization}
  Given $\alpha$, $0\leq \alpha<n$, and $p$, $1<p<n/\alpha$, and Young
  functions $A$ and $B$, define two new Young functions
  $\Phi(t)=A(t^{1/p})$ and $\Psi(t)=B(t^{1/p'})$.  Then for any
  non-negative, locally integrable functions $w_1,\,w_2$, the pair of
  factored weights
\[ (\tilde{u},\tilde{v})=(w_1(M_{\Psi,\alpha}w_2)^{1-p}, (M_{\Phi,\alpha}w_1)w_2^{1-p} ) \]
satisfies
\begin{equation} \label{eqn:rf-ap}
 \sup_Q |Q|^{\alpha/n}\|\tilde{u}^{1/p}\|_{A,Q}\|\tilde{v}^{-1/p}\|_{B,Q} <
\infty.
\end{equation}
\end{lemma}

For clarity and completeness, we include the short proof. 

\begin{proof}
By the definition of the Orlicz maximal operators and the Luxemburg
norm,
\begin{multline*}
 \|\tilde{u}^{1/p}\|_{A,Q} \approx \|\tilde{u}\|^{1/p}_{\Phi,Q}  \\ = 
\|w_1(M_{\Psi,\alpha}w_2)^{1-p}\|^{1/p}_{\Phi,Q} \leq
|Q|^{-\alpha/(np')}\|w_1\|^{1/p}_{\Phi,Q} \|w_2\|_{\Psi,Q}^{-1/p'}.
\end{multline*}
In exactly the same way we have that 
 \[ \|\tilde{v}^{-1/p}\|_{B,Q}\leq |Q|^{-\alpha/(np)}\|w_1\|^{-1/p}_{\Phi,Q}
\|w_2\|_{\Psi,Q}^{1/p'}. \]
The desired conclusion follows at once.
\end{proof}

To state the next result,  recall that the
Fefferman-Stein sharp maximal operator is defined by
\[ M^\#f(x) = \sup_{Q\ni x} \dashint_Q |f(y)-a_f(Q)|\,dy. \]
Given $\delta$, $0<\delta<1$, let $M_\delta^\#f(x)=
M^\#(|f|^\delta)(x)^{1/\delta}$. 

\begin{lemma}[\mbox{\cite[Theorem~9.10]{CMP4}}] \label{lemma:sharp-ineq}
Let $S$ and $T$ be a pair of operators such that for all $\delta$, $0<\delta<1$,
$M_\delta^\#(Tf)(x)\leq c_\delta Sf(x)$.  Then for all $p$,
$1<p<\infty$, if the pair $(u,v)$ satisfies
\[ \sup_Q \|u^{1/p}\|_{A,Q}\|v^{-1/p}\|_{p',Q} < \infty, \]
where $A(t)=t^p\log(e+t)^{p-1+\nu}$, $\nu>0$, then 
\[ \|Tf\|_{L^{p,\infty}(u)} \leq \|Sf\|_{L^{p,\infty}(v)}. \]
\end{lemma}

To apply this lemma, we will need the following sharp function
inequalities:
\begin{itemize}
\item (Adams \cite{adams75}) For $0<\alpha<n$ and $0<\delta\leq 1$, 
\begin{equation} \label{eqn:sharp1}
M_\delta^\#(I_\alpha f)(x) \leq cM_\alpha
  f(x).
\end{equation}

\item (\'Alvarez and P\'erez \cite{alvarez-perez94}) If $T$ is a
  Calder\'on-Zygmund singular integral operator, then
\begin{equation} \label{eqn:sharp2}
M_\delta^\#(Tf)(x)\leq Mf(x).
\end{equation}

\item (P\'erez \cite{perez97})  If $T$ is a Calder\'on-Zygmund singular integral operator and $b\in
  BMO$, then for $0<\delta<\epsilon<1$, 
\begin{equation} \label{eqn:sharp3}
 M_\delta^\#([b,T]f)(x) \leq
c\|b\|_{BMO}\left(M_\epsilon(Tf)(x)+M_{L\log L}f(x)\right).
\end{equation}

\item (\cite{CF}) For $0<\alpha<n$, $0<\delta\leq 1$, and $b\in BMO$, 
\begin{equation} \label{eqn:sharp4}
 M_\delta^\#([b,I_\alpha]f)(x) \leq c\|b\|_{BMO}\left( I_\alpha
  f(x)+M_{L\log L,\alpha}f(x)\right).
\end{equation}

\item (\cite{cruz-uribe-martell-perez04})  
\begin{equation} \label{eqn:sharp5} 
M_\delta^\#(Mf)(x) \leq cM^\#f(x).
\end{equation}
\end{itemize}

If we combine inequalities \eqref{eqn:sharp2} and \eqref{eqn:sharp5},
we get another sharp function inequality.  Fix $0<\delta<\epsilon<1$,
and let $\sigma=\delta/\epsilon<1$.  Then
\begin{multline} \label{eqn:multi-sharp}
M_\delta^\#(M_\epsilon(Tf))(x) =
M^\#(M(|Tf|^\epsilon)^\sigma)(x)^{\frac{1}{\sigma}\frac{1}{\epsilon}} \\
= M_\sigma^\#(M(|Tf|^\epsilon))(x)^{\frac{1}{\epsilon}} 
\leq cM^\#(|Tf|^\epsilon)(x) ^{\frac{1}{\epsilon}}  \\ =
cM_\epsilon^\#(Tf)(x) \leq cMf(x).
\end{multline}

\medskip

\begin{proof}[Proof of Theorem~\ref{thm:factored-sio}]
Recall that $\Phi(t)=t\log(e+t)^{p+\delta}$.  
Let $\Phi_0(t)=t\log(e+t)^{p-1+\delta/2}$ and
$\Phi_1(t)=t\log(e+t)^{2p-1+\delta}$.   Then by a result of Carozza
and Passarelli di Napoli \cite{carozza-passarelli96} (see also
\cite[Theorem 5.26]{CMP4}), we have that for any function $h$,
\[ M_{\Phi_0}(M_{\Phi_0}h)(x) \leq cM_{\Phi_1}h(x) \]
and
\[ M(M_{\Phi_1}h)(x) \leq cM_{\Phi}h(x). \]
Similarly, recall that $\Psi(t)=t\log(e+t)^{p'+1}$.  If we let
$\Psi_0(t)=t\log(e+t)^{p'}$ and $\Psi_1(t)=t\log(e+t)^{p'-1}$, then 
\[ M(M_{\Psi_1}h)(x) \leq cM_{\Psi_0}h(x), \qquad  M(M_{\Psi_0}h)(x) \leq cM_{\Psi}h(x).\]
By Lemma~\ref{lemma:reverse-factorization}, the pair
\[ \big(w_1(M(M_{\Psi_0}w_2))^{1-p},
M_{\Phi_0}w_1(M_{\Psi_0}w_2)^{1-p}\big) \]
satisfies \eqref{eqn:rf-ap} with $A(t)=t^p\log(e+t)^{p-1+\delta/2}$
and $B(t)=t^{p'}$.  Therefore, by Lemma~\ref{lemma:sharp-ineq} and \eqref{eqn:sharp3},
\begin{align*}
\|[b,T]f\|_{L^{p,\infty}(\tilde{u})}
& = \|[b,T]f\|_{L^{p,\infty}(w_1(M_\Psi w_2)^{1-p})} \\
& \leq c \|[b,T]f\|_{L^{p,\infty}(w_1(M(M_{\Psi_0} w_2))^{1-p})} \\
& \leq c\|b\|_{BMO}\bigg(\|M_\epsilon(Tf)\|_{L^{p,\infty}(M_{\Phi_0}w_1(M_{\Psi_0}
  w_2)^{1-p})} \\
& \quad + \|M_{L\log L}f\|_{L^{p,\infty}(M_{\Phi_0}w_1(M_{\Psi_0}
  w_2)^{1-p})}\bigg).
\end{align*}

We estimate each of the final terms separately.  By
Lemma~\ref{lemma:reverse-factorization} the pair
\[ \big(M_{\Phi_0}w_1(M(M_{\Psi_1}w_2))^{1-p}, 
M_{\Phi_0}(M_{\Phi_0}w_1)(M_{\Psi_1}w_2)^{1-p}\big) \]
again satisfies \eqref{eqn:rf-ap} with $A(t)=t^p\log(e+t)^{p-1+\delta/2}$
and $B(t)=t^{p'}$.   Similarly, the pair
\[  \big(M_{\Phi_1}w_1M_{\Psi_1}w_2)^{1-p}, 
M(M_{\Phi_1}w_1)w_2^{1-p}\big) \]
satisfies \eqref{eqn:rf-ap} with $A(t)=t^p$
and $B(t)=t^{p'}$.  In particular, this pair satisfies the two-weight
$A_p$ condition.  Therefore, by Lemma~\ref{lemma:sharp-ineq} and
\eqref{eqn:multi-sharp}, and by the two-weight, weak $(p,p)$
inequality for the maximal operator,
\begin{align*}
\|M_\epsilon(Tf)\|_{L^{p,\infty}(M_{\Phi_0}w_1(M_{\Psi_0}
  w_2)^{1-p})}
& \leq c \|M_\epsilon(Tf)\|_{L^{p,\infty}(M_{\Phi_0}w_1(M(M_{\Psi_1}
  w_2))^{1-p})} \\
& \leq c
\|Mf\|_{L^{p,\infty}(M_{\Phi_0}(M_{\Phi_0}w_1)(M_{\Psi_1}w_2)^{1-p})}
\\
& \leq c
\|Mf\|_{L^{p,\infty}(M_{\Phi_1}w_1(M_{\Psi_1}w_2)^{1-p})} \\
& \leq c
\|f\|_{L^{p,\infty}(M(M_{\Phi_1}w_1)w_2^{1-p})} \\
& \leq c \|f\|_{L^{p}(M_{\Phi}w_1w_2^{1-p})} \\
& = c\|f\|_{L^p(\tilde{v})}.
\end{align*}

The estimate for the second term is simpler.  By
Lemma~\ref{lemma:reverse-factorization}, the pair
\[ \big(M_{\Phi_0}w_1(M_{\Psi_0}w_2)^{1-p},
M(M_{\Phi_0}w_1)w_2^{1-p}\big) \]
satisfies \eqref{eqn:rf-ap} with $A(t)=t^p$
and $B(t)=t^{p'}\log(e+t)^{p'}$.   Therefore, by
Lemma~\ref{thm:orlicz-weak},
\begin{align*}
\|M_{L\log L}f\|_{L^{p,\infty}(M_{\Phi_0}w_1(M_{\Psi_0}
  w_2)^{1-p})} 
& \leq c\|f\|_{L^p(M(M_{\Phi_0}w_1)w_2^{1-p})} \\
& \leq c\|f\|_{L^p(M_{\Phi}w_1w_2^{1-p})} \\
& = c\|f\|_{L^p(\tilde{v})}.
\end{align*}
\end{proof}

\begin{proof}[Proof of Theorem~\ref{thm:factored-frac}]
The proof is nearly the same as the proof of
Theorem~\ref{thm:factored-sio}, except that instead of having to
introduce the supplementary maximal operators $M_{\Psi_0}$ and
$M_{\Psi_1}$, we use the fact that $M_{\Psi,\alpha}w_2 \in A_1$ (see
\cite[Proposition~6.15]{CMP4}), so
$M(M_{\Psi,\alpha}w_2)\approx M_{\Psi,\alpha}w_2$.    Given this we
can repeat the steps of the above proof, using the appropriate sharp
function inequalities for $[b,I_\alpha]$ and $I_\alpha$.
\end{proof}

\section{Sharp examples}
\label{section:sharp-examples}

\subsection{Sharp two-weight condition for $[b,T]$}

The example that shows that in Theorem \ref{main} we cannot take
$\delta=0$ was actually constructed in \cite{perez97}.  There it was
shown that~\eqref{2weightcom} is false for the Hilbert transform when
we take the  pair of weights
$(u,M_{\Phi}u)$ where $\Phi(t)=t\log(e+t)^{2p-1}$, $p>1$ an integer.
By Lemma~\ref{lemma:reverse-factorization} this pair
satisfies~\eqref{twowcond} with $\delta=0$.

\subsection{Sharp two-weight condition for $[b,I_\al]$}
We show that we may not take $\delta=0$ in Theorem \ref{fraccomm} when
$p=q=k$ for a positive integer $k$, $1<k<n/\alpha$. In fact, We 
construct a pair of weights $(u,v)$ satisfying \eqref{twowcondfrac}, a
function $f$, and a $BMO$ function $b$, such that that the weak type
inequality
\[ u(\{x\in \R^n: |[b,I_\al]f(x)|>1\})\leq C\int_{\R^n} |f(x)|^k
v(x)\, dx, \]
does not hold for any constant $C>0$.  

Our example is similar to the
example for the Hilbert transform given above.  Let
$\Phi_k(t)=t\log(t+e)^{2k-1}$ and consider the pair
$(u,v)=(u,M_{\Phi_k,k\alpha}u)$.    The proof of 
Lemma~\ref{lemma:reverse-factorization} can be easily modified to show
that $(u,v)$ satisfies~\eqref{twowcondfrac} with $A(t)=t^{p}\log(e+t)^{2p-1}$,
$B(t)=t^{p'}\log(e+t)^{2p'-1+\delta}$, $\delta>0$.  (In fact, we can
take $B$ to be any Young function.)

To work with this pair, we express $v=M_{\Phi_k,k\alpha}u$ in a
different way.  
By an inequality of Stein (see  Wilson~\cite [Chapter 10]{Wilson}),
$$ \|f\|_{\Phi_k,Q} \leq c\dashint_Q M^{2k-1} f(x)\, dx,$$
where $M^j$ is the composition of $M$ with itself $j$ times.  It follows that 
$$M_{\Phi_k,k\al}f\leq c M_{k\al} (M^{2k-1} f).$$
On the other hand we have (see \cite [Example 5.42]{CMP4})
$$M_{k\al}(M^{2k-1}f)\leq c M_{\Phi_k,k\al}f; $$
hence,
$$M_{\Phi_k,k\al}f \approx M_{k\al}(M^{2k-1}f).$$

Now define the weight 
$$u(x)=\frac{\chi_{\R^n\backslash B(0,e^e)}(x)}{|x|^n\log |x|\log\log |x|}.$$
Then for $|x|> e^{e^e}$, calculations show 
$$M^{2k-1} u(x) \approx \frac{(\log |x|)^{2k-2}\log\log\log |x|}{|x|^n}.$$
and
$$M_{k\al}(M^{2k-1}u)(x) \approx \frac{(\log |x|)^{2k-1}\log\log\log |x|}{|x|^{n-k\al}}.$$

Define the function $f$ by
$$f(x)=\frac{\chi_{\R^n\backslash B(0,e^{e^e})}(x)}{|x|^\al(\log |x|)^2\log\log |x|}.$$
Then
\begin{multline*}
\int_{\R^n} |f(x)|^k M_{k\al}(M^{2k-1} u)(x)\, dx \\
\approx \int_{\R^n\backslash B(0,e^{e^e})} 
\frac{\log\log\log |x|}{|x|^n\log |x|(\log\log |x|)^k}\, dx<\infty.
\end{multline*}
Further, for each $x$, $0\leq I_\al f(x)<\infty.$  

Finally, let $b(x)=\log|x|$; then for $|x|>e^{e^e}$ we have
\begin{multline*}
I_\al(bf)(x) =
\int_{\R^n\backslash B(0,e^{e^e})}\frac{\log|y|}{|x-y|^{n-\al}
|y|^\al(\log|y|)^2\log\log|y|}\, dy\\
\geq  2^{\al-n}\int_{\R^n\backslash B(0,|x|)}\frac{1}{|y|^n\log|y|\log\log|y|}\, dy=\infty.
\end{multline*}
Hence,
$$u(\{x\in \R^n: |[b,I_\al]f(x)|>1\})\geq \int_{\R^n\backslash B(0,e^{e^e})} u(x)\, dx =\infty.$$

\subsection{Sharp one weight bound for $[b,I_\al]$}
We show the estimate \eqref{sharpbound} is sharp in the sense that the
exponent $(2-\alpha/n)\max(1,p'/q)$ cannot be replaced by any smaller
power.  It will suffice to prove this assuming that $p'/q\geq 1$; the
case when $p'/q<1$ follows at once by duality, using the fact that the
commutator is essentially self-adjoint (i.e.,
$[b,I_\al]^*=-[b,I_\al]$) and the fact that if $w\in A_{p,q}$, then $w^{-1}\in
A_{q',p'}$ and $[w^{-1}]_{A_{q',p'}}=[w]_{A_{p,q}}^{p'/q}$.

For each $\delta \in (0,1)$, define the weight
$w_\delta(x)=|x|^{(n-\delta)/p'}$ and the power functions
$f_\delta(x)=|x|^{\delta-n}\chi_{B(0,1)}(x)$.  A straightforward
computation show that
\[ \|f_\delta\|_{L^p(w_\delta^p)}\approx \delta^{-1/p}. \]
Further, we have that
\[ [w_\delta]_{A_{p,q}}\approx \delta^{-q/p'}\; \]
Since $w_\delta$ is a radial function, it suffices to check this for
balls centered at the origin, again a straightforward computation.

Let $b$ be the $BMO$ function $b(x)=\log|x|$.  We estimate the
commutator as follows.  For $x\in \R^n$, $|x|\geq 2$, 
\begin{align*}
[b,I_\al]f_\delta(x)
& =\int_{B(0,1)}\frac{\log(|x|/|y|)}{|x-y|^{n-\al}}  |y|^{\delta-n}\, dy \\
& = |x|^{\delta-n+\alpha}
\int_{B(0,|x|^{-1})}\frac{\log(1/|z|)}{|x/|x|-z|^{n-\al}}
  |z|^{\delta-n}\, dy \\
& \geq |x|^{\delta-n+\alpha}
\int_{B(0,|x|^{-1})}\frac{\log(1/|z|)}{(1+|z|)^{n-\al}}  |z|^{\delta-n}\, dy \\
& \geq \frac{|x|^{\delta}}{(1+|x|)^{n-\al}}|\mathbb S^{n-1}|\int_0^{|x|^{-1}} \log(1/r)r^{\delta-1}\,dr \\
& \geq \frac{c}{\delta^2|x|^{n-\al}},
\end{align*}
where $|\mathbb S^{n-1}|$ is the surface measure of the unit sphere in
$\R^n$.  (See \cite [p. 11]{CPP} for a similar calculation). 

Integrating this inequality, and using the fact that
$1/p-1/q=\alpha/n$ and $p'/q\geq 1$, we get that
\begin{multline*}
 \|[b,I_\alpha]f_\delta\|_{L^q(w_\delta^q)} \geq
c\delta^{-2}\left(\int_{\R\backslash B(0,2)}
  \frac{|x|^{(n-\delta)q/p'}}{|x|^{(n-\al)q}}\,dx\right)^{1/q} \\
=c\delta^{-2}\left(\int_{\R\backslash B(0,2)}|x|^{-\delta q/p'-n}\,dx\right)^{1/q} 
= c\delta^{-2-\frac{1}{q}} \\
=
c[w_\delta]_{A_{p,q}}^{(2-\frac{\alpha}{n})\max\left(1,\frac{p'}{q}\right)}\|f_\delta\|_{L^p(w_\delta^p)}.
\end{multline*}

Since this is true for every $\delta>0$, it follows that we cannot
take any smaller exponent in \eqref{sharpbound}.

\subsection{Sharp weighted Sobolev inequality}
We will show that the power $1/n'$ is sharp in \eqref{sobolev}.   Unlike
the previous example, since we are dealing with regular functions we
have to replace the cut-off function $\chi_{B(0,1)}$ with a smooth
function that has exponential decay.  

Fix $p,\,q$ such that $1\leq p<n$ and $\frac1p-\frac1q=\frac1n$,  and
take any $\delta\in (0,1)$.   Define the weight
\[ w_\delta(x)=|x|^{(\delta-n)/q}; \]
if $p>1$, then arguing as in the previous example we have that 
\[
[w_\delta]_{A_{p,q}}=[w_\delta^{-1}]_{A_{q',p'}}^{\frac{q}{p'}}=\delta^{-1}.
\]
If $p=1$, we also have $[w_\delta]_{A_{1,q}}\approx \delta^{-1}$ (see
\cite[Section 7]{LMPT}).  

Define $f_\delta(x)=\exp(-|x|^\delta)$; then we immediately have that
$$|\nabla f_\delta(x)|=\delta |x|^{\delta-1}\exp(-|x|^\delta).$$
Further, we have that 
\begin{multline*}
 \|w_\delta f_\delta\|_{L^q}=
\left(\int_{\R^n} \exp(-q|x|^\delta)|x|^{\delta-n}\, dx\right)^{1/q}
\\
= \left(|\mathbb S^{n-1}|\int_0^\infty e^{-qr^\delta} r^{\delta-1}\,dr\right)^{1/q}
=\left(\frac{|\mathbb S^{n-1}|}{q\delta}\right)^{1/q}=c\delta^{-1/q},
\end{multline*}
where again $|\mathbb S^{n-1}|$ is the surface measure of the unit sphere in $\R^n$.
Similarly, 
\begin{multline*}
\|\nabla f_\delta\|_{L^p(w_\delta^p)}
=\delta\left(\int_{\R^n} \exp(-p|x|^\delta)|x|^{(\frac{p}{n'}+1)\delta-n}\, dx\right)^{1/p}\\
=\delta^{1-1/p} |\mathbb S^{n-1}|^{1/p} 
\left(\int_0^\infty e^{-pu}u^{\frac{p}{n'}}\, du\right)^{1/p}=c\delta^{1-1/p}.
\end{multline*}
Combining these estimates we get 
$$\delta^{-1/q}\approx 
\|f_\delta\|_{L^q(w_\delta^q)}\leq c[w_\delta]^{1/n'}\||\nabla f_\delta| \|_{L^p(w_\delta^p)}
\approx \delta^{-1/n'} \delta^{1-1/p}=\delta^{1/n-1/p};$$
since this is true for all $\delta>0$, we see that the exponent $1/n'$
is sharp.


\begin{thebibliography}{99}

\bibitem{adams75}
D.~R. Adams,
\newblock {\em A note on {R}iesz potentials},
\newblock {Duke Math. J.} {\bf 42}, (1975) 765--778.


\bibitem{ACS} T. Alberico, A. Cianchi, and C. Sbordone, \emph{Fractional integrals and $A_p$-weights: A sharp estimate}, C.R. Acad. Sci. Paris, Ser I {\bf 347}, (2009) 1265--1270.

\bibitem{ABKP} J. Alvarez, R. Bagby, D. Kurtz, and C. P\'erez, \emph{Weighted estimates for commutators of linear operators}, Studia Math. {\bf 104}, (1993) 195--209.

\bibitem{alvarez-perez94}
J.~Alvarez and C.~P{\'e}rez,
\newblock {\em Estimates with {$A\sb \infty$} weights for various singular integral
  operators},
\newblock { Boll. Un. Mat. Ital. A (7)} {\bf 8}, (1994) 123--133.


\bibitem{B} S. Buckley, \emph{Estimates for operator norms on weighted spaces and reverse Jensen inequalities}, Trans. Amer. Math. Soc., {\bf 340} (1993), 253-272.

\bibitem{carozza-passarelli96}
M.~Carozza and A.~Passarelli Di~Napoli,
\newblock {\em Composition of maximal operators},
\newblock {Publ. Mat.}, {\bf 40}, (1996) 397--409.

\bibitem{SC} S. Chanillo, \emph{A note on commutators}, Indiana Univ. Math. J. {\bf 31}, (1982) 7--16.
  
  \bibitem{C-phd} D. Chung, \emph{Sharp estimates for the commutator of
    the Hilbert transform on weighted Lebesgue spaces},
  Ph.D. dissertation, University of New Mexico, 2009.  
  
  \bibitem{Ch} D. Chung, \emph{Sharp estimates for the commutators of
    the Hilbert, Riesz and Beurling transforms on weighted Lebesgue spaces}, Indiana U. Math. J., to appear.  

\bibitem{CPP} D. Chung, C. Pereyra, and C. P\'erez, \emph{Sharp bounds for general commutators on weighted Lebesgue spaces}, Trans. Amer. Math. Soc., to appear.

\bibitem{CRW} R. Coifman, R. Rochberg and G. Weiss, \emph{Factorization theorems for Hardy spaces in several variables}, Ann. of Math. {\bf 103}, (1976) 611--635.

\bibitem{cruz-uribe-fiorenza02}
D.~Cruz-Uribe and A.~Fiorenza,
\newblock {\em The {$A\sb \infty$} property for {Y}oung functions and weighted norm
  inequalities}, 
\newblock {Houston J. Math.}, {\bf 28}, (2002) 169--182.


\bibitem{CF} D. Cruz-Uribe and A. Fiorenza, \emph{Endpoint esimates and weighted norm inequalities for fractional integrals}, Publ. Mat. {\bf 47}, (2003) 103--131.

\bibitem{cruz-uribe-martell-perez04}
D.~Cruz-Uribe, J.~M. Martell, and C.~P{\'e}rez,
\newblock {\em Extrapolation from {$A\sb \infty$} weights and applications},
\newblock {J. Funct. Anal.}, {\bf 213}, (2004) 412--439.


\bibitem{CMP1} D. Cruz-Uribe, J.M. Martell, and C. P\'erez, \emph{Extensions of Rubio de Francia's extrapolation theorem}, Collect. Math. {\bf 57}, (2006) 195--231.

\bibitem{CMP2} D. Cruz-Uribe, J.M. Martell, and C. P\'erez, \emph{Sharp two-weight inequalities for singular integrals, with applications to the Hilbert transform and the Sarason conjecture}, Adv. Math. {\bf 216}, (2007) 647--676.

\bibitem{cruz-uribe-martell-perez2010}
D.~Cruz-Uribe, J.~M. Martell, and C.~P{\'e}rez,
\newblock { \em Sharp weighted estimates for approximating dyadic operators},
\newblock {Electron. Res. Announc. Math. Sci.}, {\bf17}, (2010).

\bibitem{CMP3} D. Cruz-Uribe, J.M. Martell, and C. P\'erez,
  \emph{Sharp weighted estimates for classical operators}, Adv. Math., to appear.

\bibitem{CMP4} D. Cruz-Uribe, J.M. Martell, and C. P\'erez,
  \emph{Weights, extrapolation and the theory of Rubio de Francia},
{Operator Theory: Advances and 
Applications, 215}, Birkhauser, Basel, (2011).  

\bibitem{CP1} D. Cruz-Uribe and C. P\'erez, \emph{Sharp two-weight, weak-type norm inequalities for singular integral operators}, Math. Res. Lett. {\bf 6}, (1999) 417--427.

\bibitem{CP2} D. Cruz-Uribe and C. P\'erez, \emph{Two-weight, weak-type norm inequalities for fractional integrals, Calder\'on-Zygmund operators and commutators}, Indiana Univ. Math. J. {\bf 49}, (2000) 697--721.

\bibitem{CP3} D. Cruz-Uribe and C. P\'erez, \emph{On the two weight problem for singular integral operators}, Ann. Sc. Norm. Super. Pisa Cl. Sci. (5) Vol I. (2002) 821--849.

\bibitem{DV} O. Dragicevic and A. Volberg, \emph{Sharp estimate of the Ahlfors-Beurling operator via averaging martingale transforms}, Michigan Math. J. {\bf 51} (2003) 415--435.

\bibitem{HLRV-P2009}
T.~P. Hyt\"onen, M.~Lacey, H. Marikainen, T. Orponen, M.~C. Reguera, E. Sawyer, and I. Uriarte-Tuero, {\em Weak and strong-type estimates for maximal truncations of Calder\'on-Zygmund operators on $A_p$ weighted spaces}, {Preprint}, (2011).

\bibitem{H} T.~P. Hyt\"onen, \emph{The sharp weighted bound for
    general Calderon-Zygmund operators}, preprint, 2010.

\bibitem{HPTV} T.~P. Hyt\"onen, C. P\'erez, S. Treil, and A. Volberg,
  \emph{Sharp weighted estimates of the dyadic shifts and $A_2$
    conjecture}, preprint, 2010. 

\bibitem{J} S. Janson, \emph{Mean oscillation and commutators of singular integral operators}, Ark. Mat. {\bf 16} (1978), 263--270.

\bibitem{Jour} J. L. Journ\'e, {\em Calder\'on-Zygmund operators, pseudo-differential operators and the Cauchy integral of Calder\'on} Lect. Notes Math. {\bf 994}, Springer-Verlag 

\bibitem{Lac} M. Lacey, \emph{Commutators with Riesz potentials in one and several parameters}, Hokkaido Math. J., {\bf 36} (2007), 175--191.

\bibitem{LMPT} M. Lacey, K. Moen, C. P\'erez, and R.H. Torres, \emph{Sharp weighted bounds for fractional integral operators}, J. Funct. Anal. {\bf 259}, (2010) 1073--1097.

\bibitem{LPR2010}
M.~Lacey, S.~Petermichl, and M.~Reguera,
{\em Sharp $A_2$ inequality for {H}aar shift operators},
{ Math. Annalen} {\bf 348} (2010), 127--141.


\bibitem{Lern1} A. Lerner, \emph{A pointwise estimate for the local sharp maximal function with applications to singular integrals}, Bull London Math. Soc. to appear (2010).


\bibitem{li2006} W. Li, 
{\em Two-weight norm inequalities for commutators of potential type integral operators},
J. Math. Anal. Appl. {\bf 322} (2006), 1215--1223. 

\bibitem{liu-lu2004}
Z. Liu  and S. Lu, 
{\em Two-weight weak-type norm inequalities for the commutators of fractional integrals},
Integral Equations Operator Theory, {\bf 48} (2004), 397--409. 

\bibitem{MW} B. Muckenhoupt and R. Wheeden, \emph{Weighted norm inequalities for fractional integrals}, Trans. Amer. Math. Soc., {\bf 192} (1974), 261-274.

\bibitem{muckenhoupt-wheeden76}
B.~Muckenhoupt and R.~L. Wheeden,
\newblock {\em Two weight function norm inequalities for the {H}ardy-{L}ittlewood
  maximal function and the {H}ilbert transform,}
\newblock {Studia Math.}, {\bf 55} (1976), 279--294.

\bibitem{neugebauer83}
C.~J. Neugebauer,
\newblock {\em Inserting {$A\sb{p}$}-weights,}
\newblock {Proc. Amer. Math. Soc.}, {\bf 87}, (1983) 644--648.

\bibitem{perez94}
C.~P{\'e}rez,
\newblock {\em Two weighted inequalities for potential and fractional type maximal
  operators,}
\newblock {Indiana Univ. Math. J.}, {\bf 43}, (1994) 663--683.



\bibitem{P1} C. P\'erez, \emph{Endpoint estimates for commutators of singular integral operators}, J. Funct. Anal. {\bf 128}, (1995) 163--185.

\bibitem{P2} C. P\'erez, \emph{On sufficient conditions for the
    boundedness of the Hardy-Littlewood maximal operator between
    weighted $L^p$-spaces with different weights}, Proc. London
  Math. Soc {\bf 71}, (1995) 135-157.

\bibitem{perez97}
C.~P{\'e}rez,
\newblock {\em Sharp estimates for commutators of singular integrals via iterations
  of the {H}ardy-{L}ittlewood maximal function},
\newblock {J. Fourier Anal. Appl.}, {\bf 3} (1997) 743--756, .



\bibitem{P3} C. P\'erez, \emph{A course on singular integrals and
    weights}, Lecture notes from a course given and the CRM (2009).
  Available at {\tt grupo.us.es/anaresba/trabajos/carlosperez/CRM-Proceedings.pdf}.


\bibitem{Pet1} S. Petermichl, \emph{The sharp bound for the Hilbert transform in weighted Lebesgue spaces in terms of the classical $A_p$ characteristic}, Amer. J. Math.  {\bf 129}, (2007)  1355--1375.

\bibitem{Pet2} S. Petermichl, \emph{The sharp weighted bound for the Riesz transforms}, Proc. Amer. Math. Soc.  {\bf 136}, (2008) 1237--1249.

\bibitem{PetVol} S. Petermichl and A. Volberg, \emph{Heating of the Ahlfors-Beurling operator: weakly quasiregular maps on the plane are quasiregular}, Duke Math. J.  {\bf 112},  (2002),   281-305.

\bibitem{S}  E. Sawyer, \emph{private communication}.

\bibitem{SW} E. Sawyer and R. Wheeden, \emph{Weighted inequalities for fractional integrals on euclidean and homogeneous spaces}, Amer. J. Math,  {\bf 114}, (1992) 813--874.

\bibitem{V} A. Vagharshakyan, {\em Recovering singular integrals from
    Haar shifts}, Proc. Amer. Math. Soc., to appear.
    
\bibitem{Wilson} M. Wilson, {\em Weighted Littlewood-Paley Theory and Exponential-Square Integrability} Lect. Notes Math. {\bf 1924}, Springer-Verlag, (2007).  

\bibitem{zygmund} A. Zygmund, {\em Trigonometric Series}, vols. I and II,
2nd ed., Cambridge University Press, London, (1959).

\end{thebibliography}
\end{document}